\documentclass[12pt,reqno]{amsart}

\usepackage{graphicx}
\usepackage{amssymb}
\usepackage{amsthm}
\usepackage{amsfonts}
\usepackage{amsmath}
\usepackage{amstext}
\usepackage{tikz}

\theoremstyle{plain}
\newtheorem{thm}{Theorem}[section]

\newtheorem*{Cauchy}{Cauchy's Theorem}
\newtheorem*{Morse}{The Morse Stability Theorem (Over $\F$)}
\newtheorem*{Equis}{The Equi-singular Deformation Theorem}
\newtheorem*{Truncation}{The Truncation Theorem}
\newtheorem{cor}[thm]{Corollary}
\newtheorem{lem}[thm]{Lemma}

\newtheorem*{Flem}{The Fundamental Lemma}

\newtheorem{lem/defi}[thm]{Lemma/Definition}
\newtheorem{defi/lem}[thm]{Definition/Lemma}
\newtheorem{prop/defi}[thm]{Proposition/Definition}

\newtheorem*{TMT}{The Classical Morse Stability Theorem}

\newtheorem{question}[thm]{Question}
\newtheorem{notation}[thm]{Notation}
\theoremstyle{definition}
\newtheorem{defi}[thm]{Definition}
\theoremstyle{remark}
\newtheorem{rem}[thm]{Remark}

\newtheorem{attention}[thm]{\emph{\textbf{Attention}}}
\newtheorem{example}[thm]{Example}
\newtheorem{conv}[thm]{\emph{\textbf{Convention}}}

\numberwithin{equation}{section}

\newcommand{\CA}{\mathcal{A}}

\newcommand{\C}{\mathbb{C}}
\newcommand{\MC}{\mathcal{C}}
\newcommand{\CTO}{\mathcal{C}_{ord}}
\newcommand{\M}{\mathbb{M}}

\newcommand{\ii}{\sqrt{-1}}

\newcommand{\PP}{\mathcal{NP}}
\newcommand{\MP}{\mathcal{P}}
\newcommand{\R}{\mathbb{R}}
\newcommand{\D}{\mathbb{D}}

\newcommand{\F}{\mathbb{F}}

\newcommand{\K}{\mathbb{K}}
\newcommand{\Q}{\mathbb{Q}}
\newcommand{\Z}{\mathbb{Z}}
\newcommand{\ve}{\varepsilon}

\newcommand{\s}{\smallskip}
\newcommand{\m}{\medskip}

\newcommand{\mcd}{\mathcal{D}}
\newcommand{\mcv}{\mathcal{V}}
\newcommand{\CF}{\mathcal{F}}
\newcommand{\arr}{\rightarrow}
\newcommand{\mt}{\mapsto}
\newcommand{\4}{\!\!\!-\!\!\!-\!\!\!-\!\!\!-\!\!\!-\!\!\!-\!\!\!\!\!}

\newcommand{\pt}{\partial}

\let\mathscr\mathcal
\makeatletter \@input{mathrsfs.sty} \makeatother

\begin{document}

\begin{abstract}The \textit{Enriched Riemann Sphere} $\C P_*^1$
is $\C P^1$ plus a set of \textit{infinitesimals}, having the
Newton-Puiseux field $\F$ as coordinates. Complex Analysis is
extended to the $\F$-\textit{Analysis} (\textit{Newton-Puiseux
Analysis}). The classical \textit{Morse Stability 
Theorem}
is also
extended; the \textit{stability idea} is used to formulate an
\textit{equi-singular deformation theorem} in
$\C\{x,y\}(=\mathcal{O}_2)$.
\end{abstract}

\title[Enriched Riemann Sphere, Morse Stability and Equi-singularity in
$\mathcal{O}_2$] {Enriched Riemann Sphere, Morse Stability \\
and Equi-singularity in $\mathcal{O}_2$}
\author[T.-C. Kuo]{Tzee-Char Kuo}

\author[L. Paunescu]{Laurentiu Paunescu}
\address{School of Mathematics and Statistics, University of Sydney,
  Sydney, NSW, 2006, Australia}
\email{tck@maths.usyd.edu.au; laurent@maths.usyd.edu.au}

\date{\today}
\maketitle

A general principle we believe in is that the study of convergent
power series in $n+1$ variables is Global Analysis of polynomials
in $n$ variables.

In this paper this is illustrated in the case $n=1$. Loosely
speaking, the classical Morse Stability Theorem, \textit{properly
reformulated} in \S\ref{ClassicalMorse}, and the stability notion
are ``transplanted" into Algebraic Curves, then applied to the
classification problem of singularities.

\s

In \S\ref{Enriched Riemann}, the Riemann sphere $\C P^1$ is
``enriched" to $\C P_*^1$ with ``infinitesimals", which are
irreducible curve-germs, and $\C$ to $\C_*$. The Newton-Puiseux
field $\F$ of convergent fractional power series is used as
coordinates, in terms of which several structures are defined.

In \S\ref{NPA}, the Cauchy Integral Theorem, Taylor expansions,
critical points, stability, $etc.$, are generalized to $\F$, as is
the classical Morse Stability Theorem.

\s

The notion of Morse stability for polynomials over $\F$ suggests a
\textit{stronger} definition for ``\textit{equi-singular
deformation}" in $\C^2$ (Compare \cite{Zariski}).

For example, in contemporary Algebraic Geometry, both deformations
\begin{equation}\label{Phamfamily}Q(x,y,t)\!:=x^4-t^2x^2y^2+y^4,\quad
P(x,y,t)\!:=x^3-y^4-3t^2xy^{2d},\;d\geq 2,
\end{equation}
are regarded as equi-singular:\;the zero sets are topologically
trivial (Milnor $\mu$-constant).

As we shall see, however, $Q$ is \textit{not} equi-singular from
our viewpoint. The hypothesis of the Equi-singular Deformation
Theorem in \S\ref{GoeThms} is not satisfied. The associated family
$\xi^4-t^2\xi^2+1$ is not Morse stable ($\xi=0$ splits into three
critical points when $t\ne 0$).

On the other hand, $P$, the Pham family (\cite{Pham}), \textit{is}
equi-singular in the \textit{sense} of
Definition\,\ref{MorseDeformation}, although the ``polar" $P_x$
splits into $x\pm ty^d$. This is explained in
Attention\,\ref{attentionPham} and Example\,\ref{PhamExample}. The
associated family $\xi^3-1$, being independent of $t$, is
obviously Morse stable in the sense of Definition
\ref{DefinitionMorseStable}. Our Equi-singular Deformation Theorem
applies.

\s

When does a given family $F(x,y,t)$, like $Q$, $P$ above, admit a
trivialization, and of what kind? An answer is given in the
\textit{Equi-singular Deformation Theorem} in
\S\ref{GoeThms},
using the notion of Morse stability. (Similar results over $\R$
were announced in \cite{kuo-pau}.)

The \textit{Truncation Theorem} at the end of
\S\ref{GoeThms} asserts that $f(x,y)$ can be equi-singularly
deformed into its ``Puiseux root truncation"
$\hat{f}_{root}(x,y)$. (We do not assume $0$ is an isolated
singularity.) This theorem is closely related to results on
sufficiency of jets, like Morse Lemma, \cite{yacoub}, \cite{bk},
\cite{BL}, \cite{laurent3}, \cite{koike}, \cite{kuo},
\cite{laurent}, \cite{laurent2}, $etc.$. Compare also the classical book
\cite{Hancock}.

\section{The Enriched Riemann Sphere}\label{Enriched 
Riemann}

 Take a holomorphic map-germ $\mathcal A : (\C,0)\rightarrow (\C
^2,0)$, $\CA(z)\ne 0$ if $z\ne 0$. The \textit{image}
\textit{set-germ}, $Im(\CA)$, or the geometric locus of $\CA$, has
a well-defined tangent line $T(\CA)$ at $0$. We call $Im(\CA)$ an
\textit{infinitesimal} \textit{at} $T(\CA)\in \C P^1$. The set of
infinitesimals is denoted by $\C P_*^1$.

The geometric locus of $z\mt (az,bz)$ is identified with $[a:b]\in
\C P^1$, hence $\C P^1\subset \C P_*^1$.

For instance, the curve-germ $x^2-y^3=0$, being the geometric
locus of $z\mt(z^3,z^2)$, is an infinitesimal at $[0:1]$. It is
``closer" to $[0:1]$ than any $[a:1]$ is, $a\ne 0$, in the sense
that its \textit{contact} \textit{order} (defined below) with
$x=0$ is higher than that between $x=ay$ and $x=0$.

As in Projective Geometry, $\C P_*^1$ is the union $\C
P_*^1=\C_*\cup \C _*'$, where
$$\C_*\!:=\{Im(\CA)\mid T(\CA)\ne [1:0]\},
 \quad \C_*'\!:=
 \{Im(\CA)\,|\,T(\CA)\ne [0:1]\}.$$

The classical Newton-Puiseux Theorem asserts that the field $\F$
of convergent fractional power series in an indeterminate $y$ is
algebraically closed. (\cite{Brieskorn}, \cite{dJ},
\cite{waerden}, \cite{walker}, \cite{whitney}.)

Recall that a non-zero element of $\F$ is a (finite or infinite)
convergent series
\begin{equation}\label{fractionalseries}
\alpha:\;\alpha(y)=a_0y^{n_0/N}+\cdots +a_iy^{n_i/N}+\cdots,\quad
n_0<n_1<\cdots,
\end{equation}where $N\in \Z^+$, $n_i\in \Z$, $0\ne a_i\in \C$. The
\textit{order}
of $\alpha$ is $O_y(\alpha)\!:=n_0/N$, $O_y(0)\!:=+\infty$.

We can assume $GCD(N, n_0, n_1, ...)=1$. In this paper we call
$m_{puis}(\alpha)\!:=N$ the \textit{Puiseux multiplicity} of
$\alpha$ (for clarity). The \textit{conjugates} of $\alpha$ are
$$\alpha_{conj}^{(k)}(y)\!:=\sum a_i \theta^{kn_i}y^{n_i/N},\quad
0\leq k\leq N-1,\quad \theta\!:=e^{2\pi \ii/N}.$$

The following $\D$ is an integral domain with quotient field $\F$
and ideals $\M$, $\M_1$:
$$\D\!:=\{\alpha\in \F\,|\,O_y(\alpha)\geq 0\},
\;\,\M\!:=\{\alpha\,|\,
O_y(\alpha)>0\},\;\,\M_1\!:=\{\alpha\,|\,O_y(\alpha)\geq 1\}.$$

Define $|\alpha|\!:=\sum 2^{-n_i/N}|a_i|(1+|a_i|)^{-1}$. Then
 $d(\alpha,\beta)\!:=|\alpha-\beta|$ is a
 \textit{metric} on $\D$. If we fix $N$, then\underline{}
$\lim_{m\arr \infty}\sum a_i(m)y^{n_i/N}=0$
\textit{iff} each $a_i(m)\arr 0$ (point-wise convergence).

\s

Given $\alpha\in \M_1$, let $\CA(z)\!:=(\alpha(z^N),z^N)$. We then
define $\alpha_*\!:=\pi_*(\alpha)\!:=Im(\CA)$, and use
$\pi_*:\M_1\arr \C _*$, a many-to-one surjective mapping, as a
\textit{coordinate system} on $\C_*$.

A coordinate system on $\C_*'$ is $\pi_*':\M_1\arr \C _*'$,
$\alpha_*\!:=\pi_*'(\alpha)\!:=Im(\CA)$, $\CA
(z)\!:=(z^N,\alpha(z^N))$.

\s

We furnish $\C_*$ (resp.\,$\C_*'$) with the quotient topology of
$\pi_*$ (resp.\,$\pi_*'$). As for the transition function in the
overlap $\C _*\cap \C _*'$, take $x=\alpha(y)$, $n_0/N=1$, we then
``solve $y$ in terms of $x$", obtaining
$y=\beta(x)\!:=b_0x+b_1x^{n_1'/{N'}}+\cdots$, where $a_0b_0=1$,
each $b_i$ is a polynomial in finitely many of\,
$(\!\sqrt[N]{a_0})^{-1}$, $a_1/{a_0}$, $a_2/{a_0}$, .... Hence the
topologies coincide in $\C_*\cap \C_*'$.

\textit{The quotient topology on} $\C P_*^1$ \textit{is
well-defined}.

\s

Let $X$, $Y\subset \R^n$ be germs of sub-analytic sets at $0$,
$X\cap Y=\{0\}$, $X\ne \{0\}\ne Y$. Define the \textit{contact
order} $\mathcal{C}_{ord}(X,Y)$ to be the smallest number $L$ (the
Lojasiewicz exponent) such that $d(x,y)\geq a\|(x,y)\|^L$, where
$x\in X$, $y\in Y$, $\| x \|=\|y\|$, $a>0$ constant (\cite{BM2}).

In particular, for $\alpha_*$, $\beta_*\in \C P_*^1$,
$\mathcal{C}_{ord}(\alpha_*,\beta_*)$ is defined;
$\mathcal{C}_{ord}(\alpha_*,\alpha_*)\!:=\infty$. Thus, in $\C_*$,
$$\mathcal{C}_{ord}(\alpha_*,\beta_*)={\max}_j\{O_y(\alpha-
\beta_{conj}^{(j)})\}={\max}_{k,j}\{O_y(\alpha_{conj}^{(k)}-
\beta^{(j)}_{conj})\}.$$

\textit{This is the contact order structure}
\textit{on} $\C
P_*^1$.

\s

The \textit{Puiseux (characteristic) pairs} of $\alpha$
(\cite{algsurface}), which describes the iterated torus knot of
the curve-germ $\alpha_*$, is denoted by $\chi_{puis}(\alpha)$ or
$\chi_{puis}(\alpha_*)$.

\s

The \textit{Enriched Riemann Sphere} is $\C P_*^1$ furnished 
with
the above structures, $\C_*$ is the \textit{enriched complex
plane}. (The Riemann-Zariski surface (\cite{Jonsson}, p.272) is
much larger than $\C P_*^1$. For example, $x=y^{\sqrt{2}}$ defines
a point in the former, but not in the latter.)

\begin{conv}\label{conv1}Throughout this paper, $\epsilon>0$ is a
sufficiently small constant, $$\label{Interval}I_\R\!:=\{t\in
\R\,|\,|t|<\epsilon\} ,\,\; I_\C\!:=\{t\in \C\,|\,|t|<\epsilon\},
\,\; I_\F\!:=\{t\in \D\,|\,|t|<\epsilon\}.$$\end{conv}

We say $\varphi(w)$ is \textit{real analytic}, $w=u+\ii\,v\in \C$,
if it is so as a function of $(u,v)\in \R^2$.

By ``$+\cdots$" we mean ``\textit{plus higher order terms}".

\section{The $\F$-Analysis (Newton-Puiseux 
Analysis)}
\label{NPA}
Given $U\subset \D$, open, and $\phi:U\arr \F$. We say $\phi$ is
\textit{Puiseux-Lojasiewicz bounded} if every $\alpha\in U$
has a
neighbourhood $\mathcal{N}(\alpha)$ with constants $K(\alpha),
L(\alpha)>0$, such that$$m_{puis}(\phi(\xi)) \leq
K(\alpha)m_{puis}(\xi), \quad O_y(\phi(\xi))\geq
-L(\alpha),\quad\;\xi\in \mathcal{N}(\alpha).$$

\textit{In this paper we only study functions which
are
Puiseux-Lojasiewicz bounded.}

\begin{defi}We say $\phi$ is \textit{differentiable} at 
$\alpha\in U$,
with
\textit{derivative} $\phi'(\alpha)\in \F$, if
$$\phi^{\,\prime}(\alpha)=
\lim[\phi(\alpha+\delta)-\phi(\alpha)]/\delta\;\,\text{as}\;\,\delta\arr0
\;\,(\delta\in \D).$$

If $\phi'(\gamma)=0$, $\gamma$ is a \textit{critical}
\textit{point}, with \textit{multiplicity}
$$m_{crit}(\gamma)\!:=\max\{k|\phi^{(i)}(\gamma)=0,\,1\leq i\leq
k\}.$$
\end{defi}

\begin{Cauchy}\label{AnalyticityTheorem}
If $\phi^{\,\prime}(\alpha)$ exists at every $\alpha\in U$, then
all derivatives $\phi^{(k)}(\alpha)$ exist,
\begin{equation}\label{Cauchy}
\oint_{z\in C}\phi(\mu)d\mu=0,\quad \phi^{(k)}(\alpha)=
\frac{k!}{2\pi \ii}\oint_{z\in C}
\frac{\phi(\mu)}{(\mu-\alpha)^{k+1}}\,d\mu,\; k\geq
0,\end{equation}where $\mu\!:=\alpha+z\delta$, $\delta\in \D$,
$d\mu\!:=\delta \,dz$, $C$ a sufficiently small contour around\,
$0\in \C$.

Moreover, $\phi$ is \textit{$\F$-analytic} in the sense that if
$\alpha+z\delta\in U$, $|z|<r$, $z\in \C$, then
\begin{equation}\label{Taylor}\phi(\alpha+z\delta)=\phi(\alpha)+\cdots
+(1/k!)\phi^{(k)}(\alpha)(z\delta)^k+\cdots,\;\,|z|<r,
\end{equation}where $m_{puis}(\phi^{(k)}(\alpha))\leq K(\alpha)$, a constant.
\end{Cauchy}
\textit{From now on, we consider a given}
$\phi:\M_1\arr \M_1$,
which \textit{extends} to a differentiable function $U\arr \D$.
Taking $\alpha, \delta\in \M_1$, $\xi\!:=z\delta$, we have the
``Taylor expansion" of $\phi$ at $\alpha$:
\begin{equation}\label{TaylorXi}\phi(\alpha+\xi)=\sum \alpha_k\xi^k,\quad \xi
\in \M_1,\quad \alpha_k\!:=(1/k!)\phi^{(k)}(\alpha)\in
\D.\end{equation}

\textit{In this paper we always assume} $\phi$ is
\textit{mini-regular in} $\xi$, say of \textit{order} $m$, 
$i.e.$,
\begin{equation}\label{regular}O_y(\alpha_m)=0,\quad O_y(\alpha_k)+k\geq m\;\,
\text{for}\;\, 0\leq k\leq m-1.
\end{equation}
Thus, by the Newton-Puiseux Theorem, $\phi$ has $m$ roots in
$\M_1$,
\begin{equation}\label{zerooff}Z(\phi)\!:=\{\zeta\in
\M_1\,|\,\phi(\zeta)=0\}\!:=\{\zeta_1,..., \zeta_d\},\quad m=\sum
m_i,\end{equation}where $m_i\!:=m(\zeta_i)$ is the multiplicity of
$\zeta_i$. Of course $\phi$ has $m-1$ critical points in $\M_1$.

\begin{defi} In $\M_1$,
define $\mu\sim_\phi \nu$ \textit{iff} either $\mu=\nu$ or else
$$O_y(\mu-\nu)>O_y(\mu-\zeta_i)=O_y(\nu-\zeta_i),\quad 1\leq i\leq d.$$

The equivalence class of $\mu$ is denoted by $\mu_\phi$. The
\textit{height} of $\mu_\phi$ is
$$h(\mu_\phi)\!:=\max\{O_y(\mu-\zeta_i)\,|\,1\leq i\leq d\}.$$

The quotient space is $\M_{1,\phi}\!:=\M_1/{\sim_\phi}$, with the
\textit{contact order structure}:
$$\mathcal{C}_{ord}(\mu_\phi,\nu_\phi)\!:=O_y(\mu-\nu)\;\text{if}\;
\mu_\phi\ne \nu_\phi;\;
\mathcal{C}_{ord}(\mu_\phi,\mu_\phi)\!:=\infty.$$

Let $\mu_\phi(y)$ denote $\mu(y)$ with terms $y^e$ deleted,
$e>h(\mu_\phi)$; $\mu_\phi(y)$ depends only on $\mu_\phi\in
\M_{1,\phi}$. We call $\mu_\phi(y)\in \M_1$ the \textit{canonical
coordinate} of $\mu_\phi\in \M_{1,\phi}$; $\mu_\phi$ and
$\mu_\phi(y)$ are often \textit{identified}.

The \textit{Puiseux pairs} of $\mu_\phi\in \M_{1,\phi}$ is
$\chi_{puis}(\mu_\phi)\!:=\chi_{puis}(\mu_\phi(y))$.\end{defi}

\begin{example}. Take $\phi(\xi)\!:=\xi^2-2y^3$,
$\mu(y)\!:=y^{3/2}+y^{7/4}$. Then $h(\mu_\phi)=3/2$,
$\mu_\phi(y)=y^{3/2}$, $\chi_{puis}(\mu)=\{3/2, 7/4\}$,
$\chi_{puis}(\mu_\phi)=\{3/2\}$. (We also call $\mu_\phi(y)$ the
$\phi$-\textit{truncation} of $\mu(y)$.)\end{example}

The \textit{tree-model} of $\phi$ defined in \cite{kuo-lu} is our
$\M_{1,\phi}$ \textit{without} the structures. See \S
\ref{treemodel}.

\begin{defi}In $\M_{1,\phi}$ define $\xi_\phi\sim_{bar}
\eta_\phi$ \textit{iff} either $\xi_{\phi} =\eta_{\phi}$, or else
$$ h(\xi_\phi)=h(\eta_\phi)=\MC_{ord}(\xi_\phi,\eta_\phi).$$

An equivalence class is called a \textit{bar} (as in
\cite{kuo-lu}). The \textit{bar space} is the quotient
$$Bsp(\M_{1,\phi})\!:=\M_{1,\phi}/{\sim_{bar}}.$$The bar containing $\xi_\phi$
is denoted by
$B(\xi_\phi)$, having \textit{height}
$h(B(\xi_\phi))\!:=h(\xi_\phi)$; $B_h(\phi)$ denotes a bar of
height $h$. The \textit{Lojasiewicz exponent} of $\phi$
\textit{at} $\xi$, or \textit{at} $\xi_\phi$, or \textit{on}
$B(\xi_\phi)$, is
\begin{equation}\label{loj}
L_\phi(\xi)\!:=L(\xi_\phi)\!:=L(B(\xi_\phi))\!:=O_y(\phi(\xi)).\end{equation}
\end{defi}

That (\ref{loj}) is well-defined is an easy consequence of the
following:
$$\phi(\xi)=unit\cdot \prod(\xi-\zeta_i)^{m_i}, \quad O_y(\phi(\xi))=
\sum m_i\,
O_y(\xi-\zeta_i).$$

\s

An \textit{important special case} is $\mcv\!:=\M_{1,id}$ when
$\phi=id:\xi\mt \xi$. Here $Z(id)=\{0\}$.

We call $\mcv$ the \textit{value} \textit{space}, and
$0_\mcv\!:=0_{id}$ the ``zero" element.

If $\mu(y)=uy^h+\cdots$, $u\ne 0$, then $\mu_{id}$ is
\textit{completely determined} by the pair $(u,h)$. Hence, if
$h<\infty$, there is a \textit{unique} bar of height $h$,
$B_h(id)=\{(u,h)\,|\,u\ne 0\}$; this is a copy of $\C-\{0\}$.

For $h=\infty$, we have a singleton $B_\infty(id)=\{0_\mcv\}$.

\begin{defi}
The \textit{valuation} \textit{function}, $val$, also written as
$val_\phi$ for \textit{clarity}, is$$val\!:=val_\phi:\M_{1,\phi}
\arr \mcv,\quad val(\xi_\phi)\!:=val_\phi(\xi_\phi)\!:=
\phi(\xi)_{id}.$$

If $\phi'(\gamma)=0$, $\gamma\in \gamma_\phi$, $\gamma_\phi$ is a
\textit{critical point} of $val_\phi$. The
\textit{multiplicity},
$m_{crit}(\gamma_\phi)$, is the total number of such $\gamma\in
\gamma_\phi$.

The subspace $\textit{C}(val_{\phi})$ of critical points is
displayed as
\begin{equation}\label{listingC}
\textit{C}(val_{\phi})\!:=\{\gamma_{1,\phi},...,\gamma_{p,\phi}\},\quad \sum
m_{crit}(\gamma_{j,\phi})=m-1.\end{equation}
\end{defi}

\textit{Example}. Take $\phi(\xi)=\xi^4(\xi-y)^5$, $\varepsilon\ne
0$, $h>1$. Then $val_\phi(\varepsilon
y^h+\cdots)=(-\varepsilon^4,4h+5)$, $
val_\phi((1+\varepsilon)y+\cdots)=(\varepsilon^5(1+\varepsilon)^4,9)$,
$val_\phi(0)=val_\phi(y)=0_\mcv$.

\s

\begin{defi}Let $\phi$ be as in (\ref{TaylorXi}),
(\ref{regular}). Take
$$\Phi(\xi,t)\!:=\sum A_i(t)\xi^i\in \F\{\xi,t\},\quad \Phi(\xi,0)=
\phi(\xi).$$

We call $\Phi(\xi,t)\!:=\phi_t(\xi)$ an $\F$-\textit{analytic
deformation} of $\phi$.\end{defi}

\textit{In this paper, we always assume} 
$A_i(t)\in \D\{t\}$,
\textit{and $\Phi$ is mini-regular:}
$$O_y(A_i(t))+i\geq m, \quad 0\leq i\leq m.$$

\s

In $\M_1\times I_\F$, define $(\mu,t)\sim_\Phi(\nu,t')
\;\,\text{\textit{iff}}\;\, t=t',\; \mu_{\phi_t}=\nu_{\phi_t}$.

The quotient space is $\M_1\times_\Phi I_\F\!:=\M_1\times
I_\F/{\sim_\Phi}$, with valuation function
$$val_\Phi: \M_1\times_\Phi I_\F\rightarrow \mcv,\quad (\xi_{\phi_t},t)\mt
[\phi_t(\xi)]_{id}.$$

The subspace of critical points is
$$\textit{C}(val_{\Phi})\!:=
\{(\gamma_{\phi_t},t)\in \M_1\times_\Phi I_\F|\,\gamma_{\phi_t}\in
\textit{C}(val_{\phi_t})\}.$$

The \textit{bar space} $Bsp(\M_1\times _\Phi I_\F)$ is similarly
defined.

\begin{defi}\label{MorseDeformation}We say $\Phi$ is \textit{almost
Morse stable}
if there exists a homeomorphism
\begin{equation}\label{MorseStableDefi}
\tau:\textit{C}(val_{\phi})\times I_\F\arr \textit{C}(val_{\Phi}),
\quad (\gamma_{\phi}, t)\mt
(\tau_t(\gamma_{\phi}),t),\end{equation}where
$L(\tau_t(\gamma_\phi))=L(\gamma_\phi)$.

We say $\Phi$ is \textit{Morse stable} if, in addition, the
following holds.

Given $\gamma_\phi$, $\gamma'_\phi\in \textit{C}(val_{\phi})$. If
$B({\gamma}_\phi)=B({\gamma'}_\phi)$ and $val_{\phi}(\gamma_\phi)=
val_{\phi}(\gamma'_\phi)$, then
\begin{equation}\label{MorseCriticalValue}val_{\phi_t}(\tau_t(\gamma_{\phi}))
=val_{\phi_t}(\tau_t(\gamma'_{\phi})),\quad
B({\tau_t(\gamma}_\phi))=B(\tau_t({\gamma'}_\phi)).\end{equation}
\end{defi}

\begin{example}\label{PhamExample}For $P(x,y,t)$ in
(\ref{Phamfamily}), $\Phi(\xi,t)\!:=\xi^3-3t^2y^{2d}\xi-y^4$,
$val_{\phi_t}$ has a \textit{unique} critical point
$\gamma_{\phi_t}=0$ in $\M_{1,\phi_t}$,
$m_{crit}(\gamma_{\phi_t})=2$. With $\tau_t=id$, $\Phi$ \textit{is
Morse stable}.

\textit{Attention: $\phi_t$ has two critical points
$\xi=\pm ty^d$
in $\M_1$, but $val_{\phi_t}$ has only one in
$\M_{1,\phi_t}$.}

On the other hand, $\Psi(\xi,t)\!:=\xi^3+ty^3-y^4$ is \textit{not
almost Morse stable}.
\end{example}
Let $\mcv\times_{Loj} Bsp(\M_{1,\phi})\!:=\{(v,B)\,|\,v\in
\mcv,\,L(v)=L(B)\}\subset\mcv\times Bsp(\M_1,\phi)$. We may call
this a ``Lojasiewicz fiber product". Define $\mcv\times_{Loj}
Bsp(\M_1\times_\Phi I_\F)$ similarly.

\begin{Morse}Suppose $\Phi$ is Morse stable. Then there exist
$t$-level preserving homeomorphisms (deformations) $\mcd_\Phi$,
$\mcd_\mcv$ such that the diagram

\[
\begin{array}{cccc}

\M_{1,\phi}\times I_\F & \stackrel{Val_{\phi}\;\;}
{\4\longrightarrow} & \mcv\times_{Loj} Bsp(\M_{1,\phi})\times I_\F

\\

\Big\downarrow\vcenter{\rlap{$\mcd_\Phi$}} & & \Big\downarrow\vcenter{%
\rlap{$\mcd_\mcv$}}

\\

{\M_1\times_\Phi I_\F} &
\stackrel{Val_{\Phi}\;\;}{\4\longrightarrow}&
 \mcv\times_{Loj} Bsp(\M_1\times_\Phi I_\F)

\end{array}
\]
is commutative, where
$$Val_{\phi}(\mu_\phi,t)\!:=(val_{\phi}(\mu_\phi),B({\mu}_\phi),t),\;
Val_{\Phi}
(\mu_{\phi_t},t)\!:=(val_{\phi_t}(\mu_{\phi_t}),B({\mu}_{\phi_t}),t).$$

The structures are preserved. That is, if
$\mcd_\Phi(\mu_\phi,t)\!:=(\mu_{\phi,t},t)$, $\mu_{\phi,t}\in
\M_{1,\phi_t}$, then
\begin{equation}\label{structurepreserved}h(\mu_{\phi,t})=h(\mu_\phi),
\;\chi_{puis}(\mu_{\phi,t})=\chi_{puis}(\mu_\phi),\;
\CTO(\mu_{\phi,t},\nu_{\phi,t})=
\CTO(\mu_\phi,\nu_\phi);\end{equation}and if we write
$\mcd_\mcv(v,B,t)\!:=(v_t,B_t,t)$, then $L(v_t)=L(v)$,
$h(B_t)=h(B)$.
\end{Morse}

\section{The Equi-singular Deformation Theorem}
\label{GoeThms}

A \textit{real analytic} map-germ $\rho: ([0,\infty),0)\rightarrow
(\C^2,0)$ is called an \textit{analytic arc} (\cite{milnor}); we
call the \textit{image set-germ}, $Im(\rho)$, a 
\textit{geo-arc}.
The \textit{complexification} of $\rho$ is
$\rho_{_\C}(z)\!:=\rho(z)$, $z\in \C$.

Given $k\in \Z^+$, define $\rho^{(k)}(s)\!:=\rho(s^k)$. Of course,
$Im(\rho^{(k)})=Im(\rho)$, the same geo-arc.

\s

Let $f(x,y)\in \C\{x,y\}$ be given, \textit{mini-regular} in $x$
of order $m$, $i.e.$,
$$f(x,y)=H_m(x,y)+H_{m+1}(x,y)+\cdots,\;\,H_m(1,0)\ne 0.$$Let
$\phi(\xi)\!:=f(\xi,y)$. In $\C_*$, define
$\alpha_*\sim_{f}\beta_*$ \textit{iff} either $\alpha_*=\beta_*$,
or else $$
\CTO(\alpha_*,\beta_*)>\CTO(\alpha_*,\zeta_*)=\CTO(\beta_*,\zeta_*)\;\,
\forall \,\zeta\in Z(\phi).$$

The equivalence class of $\alpha_*$ is denoted by $\alpha_{*/f}$.
(If $\zeta\in Z(\phi)$, $\zeta_{*/f}\!:=\{\zeta_*\}$.)

Call $\alpha_\phi(y)$, and any one of the conjugates
$\alpha_{\phi,conj}^{(k)}(y)$, a \textit{canonical coordinate} of
$\alpha_{*/f}$.

We say $\alpha_{*/f}$, $\beta_{*/f}$ are \textit{bar equivalent}:
$\alpha_{*/f}\sim_{bar} \beta_{*/f}$, if they have canonical
coordinates $\alpha_\phi$, $\beta_\phi$ respectively, such that
$\alpha_\phi\sim _{bar}\beta_\phi$.

\begin{example} Take $f(x,y)\!:=x^2-2y^3$. The curve-germs
$x^2-y^3=0$ and $(x^2-y^3)^2-xy^5=0$ are $\sim_f$ equivalent,
hence define a same point in $\C_{*/f}$, with canonical
coordinates $\pm y^{3/2}$.\end{example}

\begin{defi}The \textit{valuation} \textit{function} on the
quotient space
$\C_{*/f}\!:=\C_*/{\sim_f}$ is
$$val_{*/f}:\C_{*/f}\rightarrow \mcv_{\Z},\quad \alpha_{*/f}
\mt [\prod_{k=1}^N\phi
(\alpha_{\phi,conj}^{(k)}(y))]_{id},\quad
N\!:=m_{puis}(\alpha_\phi),$$where $\mcv_{\Z}\!:=\{0_{\mcv}\}\cup
\{(u,h)|u\ne 0,\,h\in \Z^+\}$, a subspace of $\mcv$.

If $\phi'(\gamma)=0, \gamma\in \M_1$, $\gamma_{*/f}$ is called a
\textit{critical point} of $val_{*/f}$, the
\textit{multiplicity},
$m_{crit}(\gamma_{*/f})$, is the total number of $\mu\in \M_1$,
counting multiplicities, such that $\phi'(\mu)=0$,
$\mu_{*/f}=\gamma_{*/f}$.

The subspace of critical points of $val_{*/f}$ is denoted by
$\textit{C}(val_{*/f})$.

In $\C_{*/f}$, define \textit{height} by
$h(\alpha_{*/f})\!:=h(\alpha_\phi)$, and \textit{contact order} by
$$\CTO(\alpha_{*/f},\beta_{*/f})\!:
=\CTO(\alpha_*,\beta_*)\;\textit{if}\;\,\alpha_{*/f}\ne\beta_{*/f};\;\,
\CTO(\alpha_{*/f},\alpha_{*/f})\!:=\infty.$$ The \textit{Puiseux
pairs} are
$\chi_{puis}(\alpha_{*/f})\!:=\chi_{puis}(\alpha_\phi)$.
\end{defi}

Let $F(x,y,t)\!:=F_t(x,y)\!:=\sum_{i+j\geq m}c_{ij}(t)x^iy^j\in
\C\{x,y,t\}$ be a given deformation of $f(x,y)$, $i.e.$,
$F_0(x,y)=f(x,y), F_t(0,0)\equiv 0$. In $\C_*\times I_\C$ define
$(\alpha_*,t)\sim_F(\beta_*,t')$ \textit{iff} $t=t'$ and
$\alpha_{*/{F_t}}=\beta_{*/{F_t}}$. The quotient space is
$\C_*\times_{F}I_\C\!:=\C_*\times I_\C/{\sim_F}$.
\begin{Equis}\label{homeoH}Suppose the deformation $\Phi(\xi,t)\!:=F(\xi,y,t)$
is almost
Morse stable. Then there exists a $t$-level preserving
homeomorphism
\begin{equation}\label{homeoh}H:(\C^2\times I_\C, 0\times I_\C)\arr
(\C^2\times I_\C,0\times I_\C), \quad ((x,y),t)\mt
(H_t(x,y),t),\end{equation}which is real bi-analytic outside
$\{0\}\times I_\C$. The following hold.
\begin{enumerate}

\item $F(H_t(x,y),t)=f(x,y)$, $t\in I_\C$, $i.e.$, $F(x,y,t)$ is
``trivialized" by $H$.

\s

\item There exists $c>0$, $c\leq \|H_t(x,y)\|/\|(x,y)\|\leq 1/c$,
$t\in I_\C$.

\s

\item If $\rho(s)$ is an analytic arc, then $H_t(\rho^{(k)}(s))$
is real analytic in $(s,t)$ (called a ``geo-arc wing") for some
$k\in \Z^+$. In particular, $H_t$ is geo-arc analytic in the sense
that it carries geo-arcs to geo-arcs.

\s

\item Take $\alpha_{*/f}\in \C_{*/f}$, and any $\rho$ such that
$Im(\rho_{_\C})\in \alpha_{*/f}$. The geo-arc $Im(H_t\circ \rho)$
is contained in a unique curve-germ, say $\delta_*\in \C_*$. Then
$\eta_t(\alpha_{*/f})\!:=\delta_{*/F_t}$ is independent of the
choice of $\rho$. Hence $\eta_t:\C_{*/f}\rightarrow \C_{*/F_t}$ is
well-defined,
\begin{equation}\label{etastar}\eta_*: \C_{*/f}\times
I_\C\rightarrow \C_*\times _F I_\C,\quad (\alpha_{*/f},t)\mt
(\eta_t(\alpha_{*/f}),t),\end{equation}is a homeomorphism,
preserving height, contact order, and Puiseux pairs.

\s

\item If $\gamma_{*/f}\in \textit{C}(val_{*/f})$ then
$\eta_t(\gamma_{*/f})\in \textit{C}(val_{*/{F_t}})$,
$m_{crit}(\gamma_{*/f})=m_{crit}(\eta_t(\gamma_{*/f})).$\end{enumerate}

\m

Assume $\Phi(\xi,t)$ is Morse stable. Take $\gamma_{*/f}$,
$\gamma_{*/f}'\in \textit{C}(val_{*/f})$,
$\gamma_{*/f}\sim_{bar}\gamma'_{*/f}$. Then
\begin{equation}\label{valuationpreserve}
val_{*/f}(\gamma_{*/f})=val_{*/f}(\gamma_{*/f}')
\;\,\text{implies}\;\,
val_{*/F_t}(\eta_t(\gamma_{*/f}))=val_{*/F_t}(\eta_t(\gamma_{*/f}')),
\end{equation}where
$\eta_t(\gamma_{*/f})\sim_{bar}\eta_t(\gamma_{*/f}')$.
\end{Equis}

Now let us consider $f(x,y)$, $Z(\phi)\!:=\{\zeta_1,...,\zeta_d\}$
as in (\ref{zerooff}), $m_i\!:=m(\zeta_i)$,
$$f(x,y)=u(x,y)\cdot\prod_{i=1}^d(x-\zeta_i(y))^{m_i},\quad u(0,0)\ne 0.$$
Let $e_i\!:=\max_{j\ne i}\{O_y(\zeta_i-\zeta_j)\}$. Let
$\hat{\zeta}_i(y)$ denote $\zeta_i(y)$ with all terms $y^e$
deleted, $e>e_i$.

\begin{defi}The \textit{Puiseux root truncation} of $f(x,y)$ 
is,
by definition,
$$
\hat{f}_{root}(x,y)\!:=
\begin{cases} \prod_{i=1}^d(x-\hat{\zeta}_i(y))^{m_i}& \text{if}\;\,
d>1,\\
(x-\zeta_1)^{m_1} & \text{if}\;\, d=1.
\end{cases}
$$

Let $R_i(y)\!:=\zeta_i(y)-\hat{\zeta}_i(y)$ (the remainder). Take
$u(x,y,t)$ (a deformation of unit),
$$u(0,0,t)\ne 0,\quad u(x,y,0)=u(x,y),\quad u(x,y,1)=1.$$
The \textit {Puiseux root deformation} of $f(x,y)$ is, by
definition,
$$F_{root}(x,y,t)\!:=u(x,y,t)\cdot {\prod}_{i=1}^d[x-\zeta_i(y)+tR_i(y)]^{m_i}
\in \C\{x,y,t\}.$$
\end{defi}

Note that $\hat{f}_{root}(x,y)\in \C\{x,y\}$, since it is
invariant under the conjugations; if $d>1$ then
$\hat{f}_{root}(x,y)$ is a polynomial. The following theorem is
proved at the end of \S \ref{proofEquiSingThm}.
\begin{Truncation}\label{essential}The Puiseux root deformation
$F_{root}(x,y,t)$ is Morse
stable. In particular,
 $f(x,y)$ and $\hat{f}_{root}(x,y)$ are geo-arc analytically
equivalent.
\end{Truncation}

\begin{example}For a weighted expansion
$f(x,y)\!:=W_d(x,y)+\cdots$, if $W_d$ is non-degenerate then
$\hat{f}_{root}(x,y)=W_d(x,y)$, $f(x,y)$ is geo-arc analytically
equivalent to its initial form.\end{example}

Next, $g(x,y)\!:=(x^2-y^4)^2-y^{10}+\cdots$ has Puiseux roots
$x=\pm y^2 \pm \frac{1}{2}y^3+\cdots$,
$$\hat{g}_{root}(x,y)=(x^2-y^4)^2-\frac{1}{4}y^6[(x-y^2)^2+(x+y^2)^2]+
\frac{1}{16}y^{12}.$$
Note that $\hat{g}_{root}(x,y)$ is \textit{not} obtained by
deleting certain terms of $g(x,y)$.

\m

\begin{rem}We call $\alpha_{*/f}$ an ``$f$-\textit{blurred}"
\textit{infinitesimal}: Points of $\C_*$ equivalent under $\sim_f$
are \textit{no longer distinguishable}\,--
``\textit{f}-blurred".
The notion of blurring plays a vital role in this paper. For
example, the Pham deformations
$$P_{even}(x,y,t)=x^3-y^4+3txy^{2d},\quad
P_{odd}(x,y,t)=x^3-y^4+3txy^{2d+1},$$where $d\geq 2$, are regarded
by some experts as substantially different, since the polars are
very different when $t\ne 0$. To us, however, there is only one
critical point -- a ``\textit{blurred} \textit{polar}" -- in
either case, of multiplicity $2$. (Compare
Attention\,\ref{attentionPham}). The above theorem
applies.\end{rem}

\begin{rem}The above (\ref{valuationpreserve}) says that the
family $\{val_{*/F_t}\}$ on the spaces $\{\C_{*/F_t}\}$ is Morse
stable, the deformation of the critical points being given by
$\eta_t$. Hence, as in the classical case, we can construct a
trivialization of the family, like $(D_t,d_t)$, in
\S\,\ref{ClassicalMorse}.

However, we are not saying that this trivialization coincides with
$\eta_*$. We believe it would be too good (too strong) for this to
be true.
\end{rem}

\section{Proof Of Cauchy's Theorem}\label{proofof cauchy}

Let $\delta\in \D$ be given, and fixed. As $\phi$ is
Puiseux-Lojasiewicz bounded, we can write
$$y^L\phi(\alpha+z\delta)=\sum_{i=0}^\infty
c_{\delta,i}(z)y^{n_i/K},\quad 0<K\leq n_0<\cdots,$$ where $L$,
$K$ are constants, $|z|$ sufficiently small.

Take an increment $\Delta z$ and compute the derivative. We find
\begin{equation}\label{deltaderivative}y^L\phi'(\alpha+z\delta)=
\frac{1}{\delta}\sum
c_{\delta,i}'(z)y^{n_i/K}.\end{equation}In particular,
$c_{\delta,i}'(z)$ exists, $c_{\delta,i}(z)$ is holomorphic,
$$c_{\delta,i}(0)=\frac{1}{2\pi \ii}\oint_{z\in C}
\frac{c_{\delta,i}(z)}{z}\,dz,
\quad \phi(\alpha)=\frac{1}{2\pi \ii}\oint_{z\in
C}\frac{\phi(\mu)}{\mu-\alpha}\,d\mu,$$where
$\mu\!:=\alpha+z\delta$. Then, as in Complex Analysis,
(\ref{Cauchy}) follows.

Next we show (\ref{Taylor}). Take $\delta=1$ in
(\ref{deltaderivative}), and then set $z=0$. We have
$$y^L\phi'(\alpha+z)=\sum c_{1,i}'(z)y^{n_i/K}, \quad c_{1,i}'(0)=
\frac{1}{\delta}c_{\delta,i}'(0).$$

Applying the same argument to higher derivatives, we have
$$c_{1,i}^{(k)}(0)=\delta^{-k}c_{\delta,i}^{(k)}(0),\quad k\geq 1.$$Hence,
$$y^L\phi(\alpha+z\delta)=\sum_i\sum_k
\frac{1}{k!}c_{\delta,i}^{(k)}(0)z^ky^{n_i/K}
=\sum_k \frac{1}{k!}\phi^{(k)} (\alpha)\cdot (z\delta)^k.$$

Take $\delta=y^h$, $h\in \Q^+$. If $m_{puis}(\phi^{(k)}(\alpha))$
were unbounded, then $m_{puis}(\phi(\alpha+cy^h))=\infty$ for
generic $c$, a contradiction. This completes the proof.
\begin{cor}Suppose $\phi:U\arr \D$ is $\F$-analytic, $0\in U$, $\phi(0)=0$.
Then there exist a holomorphic germ $f(x,y)$ and $N\in \Z^+$
 such that $\phi(\xi)=f(\xi,y^{1/N})$, $\xi\in \M$.
\end{cor}
\begin{proof}Take $\alpha=0$, $S\!:=\sup_k\{m_{puis}(\phi^{(k)}(0))\}$,
 $N\!:=S!$. All $\phi^{(k)}(0)(y^N)$ in (\ref{Taylor}) are
\textit{integral} power series of $y$, $f(x,y)\!:=\sum
(1/{k!})\phi^{(k)}(0)(y^N) x^k$ is holomorphic $0$,
$f(0,0)=0$.\end{proof}

\section{Newton Polygon At An Infinitesimal}
\label{NewtonPolygon}

Given $\phi$ and $\alpha$. For a term $\alpha_k\xi^k\ne 0$ in
(\ref{TaylorXi}), plot a ``Newton dot" at $(k,q)$ in the
$(u,v)$-plane, $q\!:=O_y(\alpha_k)$. Consider the convex hull
generated by $\{(a_i+u,b_i+v)\,|\,u,\,v\geq 0\}$, $(a_i,b_i)$ the
Newton dots. \textit{The boundary $\PP(\phi,\alpha)$ is the
Newton
Polygon of $\phi$ at} $\alpha$. (See \cite{kuo-parusinski}.)

In the case $\phi(\xi,y)\!:=f(\xi,y)$, $f(x,y)\in \C\{x,y\}$, we
clearly have $\PP(\phi,\alpha)=\PP(\phi,\alpha_{conj}^{(k)})$.
Hence $\PP(\phi,\alpha_*)\!:=\PP(\phi,\alpha_{conj}^{(k)})$ is
well-defined. \textit{This is the Newton Polygon of $\phi$
at
$\alpha_*$}.

\m

Consider $\PP(\phi,\alpha)$. The edges and angles are
$E_i\!:=E_i(\alpha)$, $\theta_i\!:=\theta_{E_i}$, respectively,
$\theta_{i-1}<\theta_i$, $0\leq i\leq l$.
 The first edge $E_0$ is horizontal, the
last edge $E_l$ is vertical. Denote the right vertex of $E_i$ by
$V_i\!:=(m_i,q_i)$, and the straight line prolonging $E_i$ by
$\mathcal{L}(E_i)$.

The vertical edge $E_l$ is not important. We call
$E_{top}\!:=E_{l-1}$ the \textit{top} Newton edge. The left vertex
of $E_{l-1}$ is $(m_l,q_l)$, which is the last, and the highest,
vertex of $\PP(\phi,\alpha)$.

Take $E_i$, $i\leq l-1$. A dot $(k,q)\in E_i$ represents a term
$c_ky^q\xi^k$ of (\ref{TaylorXi}). Let
$$W_{E_i}(\xi,y)\!:=\sum c_ky^q\xi^k,\quad P_{E_i}(z)\!:=\sum c_kz^k\in
\C[z],$$sum taken over $(k,q)\in E_i$. We call $P_{E_i}(z)$ the
\textit{associated polynomial} of $E_i$.

 The \textit{height}, or \textit{co}-\textit{slope}, of $E_i$ is
\begin{equation}\label{coslope}h(E_i)\!:=\tan \theta_i,\;
\, 0\leq i\leq l-1;\quad
h(E_l)\!:=\tan\theta_l=\infty.\end{equation} The
\textit{Lojasiewicz} \textit{exponent} on $E_i$ is
$$L(E_i)\!:=q_i+m_i\tan\theta_i,\quad 0\leq i\leq l-1;
\quad L(E_l)\!:=\infty.$$

\begin{example}\label{IllustrationRelativeNP}(Fig.1,2)
Take $\alpha=0$. For $\phi_1(\xi)\!:= \xi^3+2y\xi^2+y^4$,
$\tan\theta_1=1$, $\tan\theta_{top}=3/2$, $P_{E_1}(z)=z^3+2z^2$,
$P_{top}(z)=2z^2+1$. For $\phi_2(\xi)\!:=\xi^3+2y\xi^2$,
$P_{top}(z)=z^3+2z^2$, $E_2=E_l$.
\end{example}
 \begin{figure}[ht]
 \begin{minipage}[b]{.5\linewidth}
   \centering
   \begin{tikzpicture}[scale=1.0]

    \foreach \x in {0,1,...,3} { \draw (\x, 0) -- (\x, 5); }

 \draw (0,0) -- (4,0);

 \draw [mark=thick](3,0.01) -- (4,0.01);

\draw [mark=thick] (3,0) -- (4,0);

 \draw [mark=thick](0,4) -- (2,1);

\draw  [mark=thick](0.01,4) -- (2.01,1);

 \draw (0.01,4) -- (0.01,5);

 \draw plot[mark=*] coordinates {(0,4) (2,1) (3,0)};

 \path (2.2,1) -- (3,0) node[midway,above] {$E_1$};

 \draw  [mark=thick](2,1) -- (3,0);

 \draw  [mark=thick](2.01,1) -- (3.01,0);

 \path (0,4) -- (2,1) node[midway,right] {$E_2$};

 \path (3.5,0) node[above] {$E_0$};

 \path (0,4.5) node[right] {$E_3$};

 \draw[dashed] (1.5,1) arc (180:125:0.5);

 \draw[dashed] (0.6,1) -- (2,1) node[midway,above] {$\theta_2$};

 \draw[dashed] (2.4,0) arc (180:135:0.6);

 \draw[dashed] (1.5,0) -- (3,0)

 node[midway,above] {$\theta_1$};
   \end{tikzpicture}
   \caption{$\PP(\phi _1,0)$} \label{fig:np1}
 \end{minipage}
 \hspace{0cm}
 \begin{minipage}[b]{.4\linewidth}
   \centering
   \begin{tikzpicture}[scale=1.0]

 \foreach \x in {0,1,...,3} { \draw (\x, 0) -- (\x,3 ); } \draw
 (0,0) -- (4,0);

 \draw plot[mark=*] coordinates {(2,1) (3,0)};

   \path (3.5,0) node[above] {$E_0$}; \path (2.3,1.5) -- (2.3,0)
   (2.6,1) node[midway,right] {$E_1$}; \path (2,2) node[right]
   {$E_2$};

    \draw[dashed] (2.4,0) arc (180:134:0.6);

    \path (2.2,0) node[above] {$\theta_1$};

     \draw (2,1) -- (3,0);

    \draw (2.01,1) -- (2.01,3);

    \draw (2,1) -- (2,3);

    \draw (2,1.01) -- (3,0.01);

    \draw (3,0.01) -- (4,0.01);

\end{tikzpicture}
   \caption{$\PP(\phi _2, 0)$} \label{fig:np2}
 \end{minipage}
 \end{figure}
\begin{notation}\label{notationequiv}Suppose $\PP(\phi,\alpha)$,
$\PP(\phi,\beta)$ have a common
 edge $E_k(\alpha)=E_k(\beta)$.

We write $E_k(\alpha)\equiv E_k(\beta)$ if they have the same
Newton dots, each represents a \textit{same} monomial term of
$\phi$. We write $\PP(\phi,\alpha)\equiv \PP(\phi,\beta)$ if this
is true for every edge.
\end{notation}

Observe that if $\alpha\not \in Z(\phi)$, then $m_l=0$, and
$$h(E_{l-1})=\max\{O(\alpha-\zeta_i)\}, \quad \zeta_i\in Z(\phi).$$

If $\alpha\sim_\phi \beta$, then $O(\alpha-\beta)>h(E_{l-1})$, and
hence $\PP(\phi,\alpha)\equiv \PP(\phi,\beta)$. Thus,
$P_{E_i(\alpha_\phi)}(z)$, $\PP(\phi,\alpha_\phi)$,
$E_i(\alpha_\phi)$, $etc.$, are all \textit{well}-\textit{defined}
(independent of the choice of $\alpha\in \alpha_\phi$).

(Similarly, $\PP(f,\alpha_{*/f})\!:=\PP(\phi,\alpha_\phi)$ is also
well-defined.)
\begin{thm}\label{KLKP}Take an edge $E_i$
of $\PP(\phi,\alpha_\phi)$, $1\leq i\leq l-1$. Take a critical
point $c$ of the associated polynomial $P_{E_i}(z)$, with
multiplicity $m_{crit}(c)$. Then there are exactly $m_{crit}(c)$
critical points of $\phi$ in $\M_1$, counting multiplicities, of
the form
\begin{equation}\label{muhat}\mu(y)=\hat{\alpha}_\phi(y)+[cy^{\tan\theta_i}+
\cdots],\end{equation}where
$\hat{\alpha}_\phi$ is $\alpha_\phi(y)$ with all terms $y^e$ (if
any) deleted, $e>\tan \theta_i$.

Take $\mu_\phi\in \textit{C}(val_{\phi})$,
$\mu_\phi\ne\alpha_\phi$. There exist a unique $E_i$ and a unique
critical point $c\ne 0$ of $P_{E_i}(z)$ such that $\mu(y)$ has the
form (\ref{muhat}).
\end{thm}

According to Convention\,\ref{conv1}, we can also write
(\ref{muhat}) as
$$\mu(y)=\alpha_\phi(y)+[cy^{\tan\theta_i}+\cdots].$$

Theorem \ref{KLKP} is known (\cite{kuo-lu}, see also
\cite{kuo-parusinski}). We use it several times in this paper.
\begin{Flem}Suppose $\Phi$ is almost Morse stable. Then
\begin{equation}\label{Fundlemma}\PP(\phi,\gamma_{\phi})=\PP(\phi_t,
\tau_t(\gamma_{\phi})),\quad
\chi_{puis}(\tau_t(\gamma_\phi))=\chi_{puis}(\gamma_\phi),\end{equation}where
$\gamma_\phi\in \textit{C}(val_{\phi})$ with deformation
$\tau_t(\gamma_\phi)\in \textit{C}(val_{\phi_t})$, as in
(\ref{MorseStableDefi}).

The canonical coordinate $\tau_t(\gamma_\phi)(y)$ of
$\tau_t(\gamma_\phi)$ is an $\F$-analytic function of $t\in I_\F$.
For each $E_i$, $i\leq l-1$, the family
$\{P_{E_i(\tau_t(\gamma_\phi))}(z)\}$ is almost Morse stable, as
defined below.

Moreover, if $\Phi$ is Morse stable, then so is the family
$\{P_{E_i(\tau_t(\gamma_\phi))}(z)\}$.
\end{Flem}
\begin{defi}\label{DefinitionMorseStable}Given a polynomial $p(z)$ and a
deformation
 $$p_t(x)\!:=a_0(t)x^n+\cdots+ a_n(t)\in \K\{t\}[x], \quad p_0(z)=p(z),
\quad a(t)\ne 0,$$where $\K\!:=\R$, $\C$, or $\F$, $t\in I_\K$.

A critical point $c_0\in \K$ of $p_0(x)$ is \textit{stable} if it
admits a \textit{continuous} deformation $c_t\in \K$ such that
$p_t'(c_t)=0$ and $m_{crit}(c_t)=m_{crit}(c_0)$. (The deformation
is then necessarily \textit{unique}.)

Consider the following conditions (where (3) is for Algebraic
Geometry):
\begin{enumerate}
\item Every critical point of $p_0(x)$ is stable.

\item If $c_0$, $c_0'$ are critical points of $p_0(x)$ and
$p_0(c_0)=p_0(c_0')$, then $p_t(c_t)=p_t(c_t')$.

\item If $p_0(c_0)=p_0'(c_0)=0$, $i.e.$, $c_0$ is a multiple root
of $p_0(z)$, then $p_t(c_t)=0$.
\end{enumerate}

\textit{We say} $\{p_t\}$ is \textit{almost Morse stable} if
 (1),
(3) hold, and \textit{Morse stable} if (2) also holds.\end{defi}
\begin{example}\label{splitCricalValue}Take $\K=\R$. For $p_t(x)=x^2(x^2+t^2)
\in \R[x]$, $0$
is a critical point of $p_0$ which splits into three critical
points in $\C$, one remains in $\R$. Thus $0$ admits a
\textit{unique} \textit{continuous} deformation $c_t\equiv 0$
\textit{in} $\R$. But $m_{crit}(c_t)$ is not constant, $0$ is
\textit{unstable}.\end{example}
\begin{proof}We use the ``\textit{edging forward argument"}
to prove
(\ref{Fundlemma}). The Tschirnhausen transformation is applied
recursively along the edges of $\PP(\phi_t,\gamma_\phi)$ (``edging
forward") in order to ``clear" all dots of $\phi_t$ lying below
$\PP(\phi,\gamma_\phi)$. If $m_l>0$, all dots to the left of the
vertical edge $E_l$ are also cleared. The details are as follows.

Consider $\PP(\phi,\gamma_\phi)$. The edges are denoted by $E_i$.
The right vertex of $E_i$ is $(m_i,q_i)$.

Let us first compare $\PP(\phi,\gamma_\phi)$ with
$\PP(\phi_t,\gamma_\phi)$. Write
$$\phi_t(\gamma_\phi+\xi)=\phi(\gamma_\phi +\xi)+P_t(\gamma_\phi+\xi),
\quad P_0(\gamma_\phi+\xi)\equiv 0.$$

The (non-zero) terms of $P_t$ are represented by dots. Some may
lie below $\PP(\phi,\gamma_\phi)$.

Suppose we already know that $P_t(\gamma_\phi+\xi)$ has no dot
below the lines $\mathcal{L}(E_j)$, $0\leq j\leq k-1$. We can then
clear the dots under the line $\mathcal{L}(E_k)$ as follows.

The left vertex of $E_{k-1}$ represents a term $ay^{q_k}\xi^{m_k}$
of $\phi(\gamma_\phi+\xi)$, which, with a Tschirnhausen
transformation, can ``swallow" all dots of $P_t$ of the form
$(m_k-1,q)$, $q\in \Q^+$.

This means the following. There exists $\beta_t(y)\in \M_1$,
$\F$-analytic in $t$, such that
\begin{enumerate}
\item $O_y(\beta_t(y))\geq \tan\theta_{k-1}$, $\beta_0(y)=0$;
\item The coefficient of $\xi^{m_k-1}$ in the Taylor expansion of
$\phi(\gamma_\phi+\beta_t+\xi)$ is \textit{independent} of $t$.
\end{enumerate}

Indeed, $\gamma_\phi+\xi\mt \gamma_\phi+\xi+\beta_t$ is the
\textit{unique} \textit{translation} (Tschirnhausen
transformation) which has the above two properties. (Attention: No
dot of $\phi(\gamma_\phi+\xi)$ has been swallowed in the process.
In this way, we have $\beta_0=0$. This property is important.)

Now, consider $\PP(\phi_t,\beta_t+\gamma_\phi)$. Because of (1),
$E_j(\gamma_\phi)\equiv E_j(\gamma_\phi+\beta_t)$, $j\leq k-2$,
(we use Notation \ref{notationequiv},) and $P_t$ still has no dot
below the line $\mathcal{L}(E_j)$, $j\leq k-1$.

Let $E_{k-1}'$, $E_k'$,...,  denote the remaining edges of
$\PP(\phi_t,\beta_t+\gamma_\phi)$. By (1),
$\theta_{E_{k-1}}=\theta_{E_{k-1}'}$.

Next we show $E_{k-1}=E_{k-1}'$. Let $V_j'\!:=(m_j',q_j')$ denote
the right vertex of $E_j'$.

Suppose $E_{k-1}'\ne E_{k-1}$. Then $m_k'\leq m_k-2$. We shall
derive a contradiction.

For generic $t\in \C$, $0$ is a root of $P_{E_{k-1}'}(z)$ of
multiplicity $m_k'$ ($m_k'\geq 0$), but when $t=0$, the
multiplicity is $m_k$. Hence, by an elementary argument, there
exists $a(t)$ such that $$\frac{d}{dz}P_{E_{k-1}'}(a(t))=0,\quad
P_{E_{k-1}'}(a(t))\ne 0,\quad \lim_{t\rightarrow 0}a(t)=0.$$

Thus, by Theorem \ref{KLKP}, $\phi_t$ has a critical point of the
form
$$\Gamma_t(y)\!:=[\beta_t(y)+\gamma_\phi(y)]+[a(t)y^e+\cdots],\;\,\quad
e\!:=\tan\theta_{E_{k-1}}.$$Hence
$h(\Gamma_{t,\phi_t})=\tan\theta_{E_{k-1}}$,
$L_{\phi_t}(\Gamma_{t,\phi_t})=q_k+m_k\tan\theta_{E_{k-1}}$, both
are \textit{independent} of $t$.

\m

Take $\mu_\phi\in \textit{C}(val_{\phi})$, $\mu_\phi\ne
\gamma_\phi$. Then take $E_i$, $c\ne 0$, for $\mu_\phi$ as in
(\ref{muhat}) with $\gamma_\phi$ replacing $\alpha_\phi$. We say
$\mu_\phi$ is of the \textit{lower kind} if $i\leq k-1$, and of
the \textit{higher kind} if $i\geq k$.

If $\mu_\phi=\gamma_\phi$, we say $\mu_\phi$ is of the
\textit{higher kind}.

\m

Take $\ve$, sufficiently small. The $\ve$-neighborhood
$\mathcal{N}_{\ve}(\gamma_\phi)$ of $\gamma_\phi$ clearly does not
contain any $\mu_\phi$ of the lower kind. On the other hand, if
$|t|$ is sufficiently small, then, by continuity,
$\Gamma_{t,\phi_t}\in \mathcal{N}_\ve(\gamma_\phi)$. Hence
$\Gamma_{t,\phi_t}\ne\tau_t(\mu_\phi)$ for any $\mu_\phi$ of the
lower kind.

If $\mu_\phi$ is of the higher kind, then
$L(\mu_\phi)=L(\tau_t(\mu_\phi))>L(\Gamma_{t,\phi_t})$. Hence
$\Gamma_{t,\phi_t}\ne \tau_t(\mu_\phi)$ for any $\mu_\phi$ of the
higher kind either. Thus we must have $E_{k-1}'=E_{k-1}$.

\s

Next we show $\theta_{E_k'}=\theta_{E_k}$. Suppose
$\theta_{E_k'}<\theta_{E_k}$. Then $m'_{k+1}\leq m_k-2$, and there
would exist $a(t)$ as above. Using the same argument we again
arrive at a contradiction.

\s

\textit{Hence the Tschirnhausen transformation 
$\xi\mt\xi+\beta_t$ clears all dots of $P_t$ below
$\mathcal{L}(E_k)$.}

\s

A recursive application of the Tschirnhausen transformations,
beginning with $k=1$, clears all dots of $P_t$ below
$\PP(\phi,\gamma_\phi)$.

\s

Let $\xi\mt \xi+B_t$ denote their composition. We then compare the
polygons:
$$\PP^{(0)}\!:=\PP(\phi,\gamma_\phi),\;
 \PP^{(1)}\!:=\PP(\phi_t,\gamma_\phi+B_t),\;
\PP^{(2)}\!:=\PP(\phi_t,\tau_t(\gamma_\phi)).$$We have just proved
$\PP^{(0)}=\PP^{(1)}$. Next we show $\PP^{(1)}=\PP^{(2)}$.

\textit{We can assume} $B_t=0$. This can be achieved by the
substitution $\xi\rightarrow \xi+B_t(y)$.

Let us write $P_{top}^{(1)}(z)\!:=P_{top}^{(0)}(z)+Q_t(z)$,
$Q_0(z)\equiv 0$, where $0$ is a critical point of
$P_{top}^{(0)}(z)$.

As $t$ varies away from $0$, this critical point cannot split into
two or more critical points of $P_{top}^{(1)}(z)$. For if it did,
the homeomorphism $\tau_t$ cannot exist.

Hence $0$ admits a \textit{unique continuous} deformation $c_t$,
$c_0=0$, which is a critical point of $P_{top}^{(1)}(z)$,
$m_{crit}(c_t)=m_{puis}(c_0)$. It follows that $c_t$ is a
\textit{simple root} of the equation
$$\frac{d^m}{dz^m}P_{top}^{(1)}(z)=0,\quad m\!:=m_{crit}(c_0),$$and
$$\tau_t(\gamma_\phi)(y)=\gamma_\phi(y)+c_ty^e,\quad
e\!:=\tan\theta_{top}^{(1)}.$$

It follows that $c_t$, $\tau_t(\gamma_\phi)(y)$ are
$\F$-\textit{analytic}. (The Implicit Function Theorem holds in
$\F$.)

\s

Let the Taylor expansion of $\phi_t$ at $\gamma_\phi$ be
$\sum_{k,q}c_{qk}(t)y^q\xi^k$. Then that at $\tau_t(\gamma_\phi)$
is
\begin{equation}\label{Taylor at Gamma}
\sum c_{kq}(t)y^q[\xi+c_ty^e]^k=\sum
c_{kq}(t)[y^q\xi^k+\cdots],\quad c_0=0.
\end{equation}

Consider $E_i^{(1)}$, $i\leq l-2$. Since $e>\tan\theta_i$, the
terms in ``$+\cdots$" are represented by dots lying
\textit{strictly} above all $E_i^{(1)}$. Hence $E_i^{(1)}\equiv
E_i^{(2)}$, $i\leq l-2$. (See Notation \ref{notationequiv}.)

Now we show $E_{top}^{(1)}=E_{top}^{(2)}$. (But not
$E_{top}^{(1)}\equiv E_{top}^{(2)}$.)

First suppose $P_{top}^{(0)}(0)\ne 0$. The left vertex of
$E_{top}^{(0)}=E_{top}^{(1)}$ lies on the vertical coordinate
axis. Hence $P_{top}^{(1)}(0)\ne 0$, $P_{top}^{(1)}(c_t)\ne 0$
($|t|$ small). But $P_{top}^{(2)}(0)=P_{top}^{(1)}(c_t)$, hence
$E_{top}^{(1)}=E_{top}^{(2)}$.

Suppose $P_{top}^{(0)}(0)=0$, $i.e.$, $0$ is a multiple root. As
$\PP^{(0)}=\PP^{(1)}$, we must have $$c_t\equiv 0,\quad
m_{crit}(c_t)=m_l=m_{puis}(\tau_t(\gamma_\phi)).$$Hence
$E_{top}^{(1)}=E_{top}^{(2)}$, and $\PP^{(1)}=\PP^{(2)}$.

\s

It is easy to see that $\chi_{puis}(\alpha_\phi)$ can be expressed
in terms of the co-slopes of the edges of $\PP(\phi,\alpha_\phi)$.
\textit{It follows that}
$\chi_{puis}(\tau_t(\gamma_\phi))=\chi_{puis}(\gamma_\phi)$. This
completes the proof of (\ref{Fundlemma}).

\m

Now assume $\Phi$ is Morse stable. We show
$\{P_{E_i(\tau_t(\gamma_\phi))}(z)\}$ is Morse stable.

Take a critical point $c$ of $P_{E_i(\gamma_\phi)}(z)$. Take $\mu$
as in (\ref{muhat}). Consider $\mu_\phi$, $\PP(\phi, \mu_{\phi})$,
$etc.$.

\s

Note that $P_{E_i(\gamma_\phi)}(z+c)=P_{E_i(\mu_\phi)}(z)$ (differ
merely by a translation).

\s

First, if $P_{E_i(\gamma_\phi)}(c)\ne 0$, then
$P_{E_i(\mu_\phi)}(0)\ne 0$. Hence $E_i(\mu_\phi)$ has its left
vertex on the vertical axis, $0$ being a critical point of
$P_{E_i(\mu_\phi)}(z)$. Then, as in the argument for
$\gamma_\phi$, $0$ admits a \textit{unique} \textit{continuous}
deformation
 which is a critical point of
$P_{E_i(\tau_t(\mu_{\phi})}(z)$ with constant multiplicity. This
says that $c$ is a stable critical point of
$P_{E_i(\tau_t(\gamma_{\phi}))}(z)$.

Suppose $P_{E_i(\gamma_{\phi})}(c)=0$, say of multiplicity $k$.
Then $0$ is a root of $P_{E_i(\mu_{\phi})}(z)$, also of
multiplicity $k$. Since
$\PP(\phi,\mu_{\phi})=\PP(\phi_t,\tau_t(\mu_{\phi}))$, $0$ is
obviously stable. Hence so is $c$.

\s

Now we show (2) in Definition \ref{DefinitionMorseStable}. Let us
write $p_t(z)\!:=P_{E_i(\tau_t(\gamma_{\phi}))}(z)$. Let $c\ne
c^{\,\prime}$ be critical points of $p_0(z)$. Take $\mu$, $\mu'$
for $E_i$, $c$ and $c^{\,\prime}$ respectively, $\mu_\phi$,
$\mu_\phi'\in \textit{C}(val_{\phi})$.

First, suppose $p_0(c)=p_0(c^{\,\prime})\ne 0$. In this case,
$$h(\mu_\phi)=h(\mu_\phi')=O(\mu_\phi-\mu_\phi')
=\tan\theta_{E_i}.$$Hence $B({\mu}_\phi)=B({\mu'}_\phi)$,
$val_{\phi}(\mu_\phi)= val_{\phi}(\mu'_\phi)$. Then, by
(\ref{MorseCriticalValue}),
$val_{\phi_t}(\tau_t(\mu_{\phi}))=val_{\phi_t}(\tau_t(\mu'_{\phi}))$.
That is,
$(p_t(c_t),\tan\theta_{E_i})=(p_t(c_t^{\,\prime}),\tan\theta_{E_i})\in
\mcv$. In particular, $p_t(c_t)=p_t(c_t^{\,\prime})$.

Suppose $p_0(c)=p_0(c^{\,\prime})= 0$. Then $c$, $c^{\,\prime}$
are multiple roots of $p_0$, and, as shown before, their
deformations remain multiple roots of $p_t$,
$p_t(c_t)=p_t(c_t^{\,\prime})= 0$.

\s

By the same argument, if $\Phi$ is almost Morse stable then so are
the families $\{P_{E_i(\tau_t(\gamma_\phi))}\}$.\end{proof}

\begin{cor}\label{CorFunLem}Let $Z(\Phi)\!:=\{(\zeta_t, t))\,|
\,\zeta_t\in Z(\phi_t)\}$. There exists a bijection
$$\mcd:Z(\phi)\times I_\F\arr Z(\Phi), \quad (\zeta,t)\mt (\zeta_t,t),\quad
\zeta_0=\zeta,$$where $t\mt \zeta_t$ is $\F$-analytic; the Newton
Polygon $\PP(\phi_t,\zeta_t)$ is independent of $t$.

Take $\zeta$, $\zeta'\in Z(\phi)$, $\gamma_\phi\in C(val_\phi)$.
Then $O_y(\zeta_t-\zeta'_t)$, $O_y(\zeta_t-\tau_t(\gamma_\phi))$
are independent of $t$.\end{cor}
\begin{proof}If $\zeta$ is a multiple root, $\zeta=\gamma_\phi$,
then $\zeta_t=\tau_t(\gamma_\phi$) is the deformation.

Otherwise, we choose $\gamma_\phi\in C(val_{\phi})$ such that
\begin{equation}\label{choosegamma}O(\zeta-\gamma_\phi)\geq O(\zeta-\mu_\phi),
\quad \forall\,
\mu_\phi\in \textit{C}(val_\phi).\end{equation}We then have
$h(\gamma_\phi)=O(\zeta-\gamma_\phi)$, and
$$\zeta(y)=\gamma_\phi(y)+[by^{h(\gamma_\phi)}+\cdots],
\quad P_{E(\gamma_\phi)}(b)=0\ne
P'_{E(\gamma_\phi)}(b).$$

(If $P'_{E(\gamma_\phi)}(b)=0$, then there would exist $\mu_\phi$
which fails (\ref{choosegamma}).) Using the Implicit Function
Theorem we can find $\zeta_t$, which is $\F$-analytic, and
$O(\zeta_t-\tau_t(\gamma_\phi))$ is constant.

\s

Given $\zeta'$. Take $\gamma'_\phi$ as in (\ref{choosegamma}).
Then $O(\tau_t(\gamma_\phi)-\tau_t(\gamma'_\phi))$ is constant, so
is $O(\zeta_t-\zeta_t')$.\end{proof}

\section{Relations Between Bars And Edges}
\label{RelationBarEdges}

Take a bar $B$, $h(B)<\infty$. Take $\beta\in\beta_\phi\in B$.
Define $\zeta_B(y)$ to be $\beta(y)$ with all terms $y^e$ deleted,
$e\geq h(B)$. Clearly, $\zeta_B(y)$ depends \textit{only} on $B$,
not on the choices of $\beta$, $\beta_\phi$.

 Take an indeterminate $z$, and write
\begin{equation}\phi(\zeta_B(y)+zy^{h(B)},y)\!:=P_B(z)y^{L(B)}+\cdots,\;\,
P_B(z)\not \equiv 0,\end{equation}where $L(B)$ was defined in
(\ref{loj}). \textit{We call} $P_B(z)$ \textit{the}
\textit{associated} \textit{polynomial of} $B$.

\s

Let $Z(P_B)$ denote the zero set of $P_B(z)$. Using the canonical
coordinates, we can identify $B$ with $\C-Z(P_B)$. Hence
$\bar{B}$, the \textit{metric space completion} of $B$, is a 
copy
of $\C$; and
$$val_\phi(\zeta_B(y)+zy^{h(B)})=(P_B(z),L(B))\in \mcv,\quad z\not 
\in Z(P_B).$$

\s

Take $\alpha\in \M_1$. If $\alpha(y)=\zeta_B(y)+ay^{h(B)}+\cdots$,
$a\in \C$, \textit{we say} $\bar{B}$ \textit{is a support of}
$\alpha$; $a$ \textit{is the} $\bar{B}$-\textit{coordinate}
\textit{of} $\alpha$, \textit{and also of} $\alpha_\phi$. Observe
that $\alpha_\phi\in B$ \textit{iff} $a\not\in Z(P_B)$.
\begin{notation}\label{support} Write
$\alpha \perp \bar{B}$ if $\alpha$ is supported by $\bar{B}$;
$Supp(\alpha)\!:=\{\bar{B}\mid\alpha\perp
\bar{B}\}$.\end{notation}

Let $\alpha$ be given. We now define $\PP_{ext}(\phi,\alpha)$ by
adding ``vertex edges" to $\PP(\phi,\alpha)$.

Take a vertex $V_i=(m_i,q_i)$ of $\PP(\phi,\alpha)$, $m_i\geq 1$,
representing a term $cy^{q_i}\xi^{m_i}$ in (\ref{TaylorXi}), where
$c\ne 0$. Take $h\in \Q^+$, $\tan\theta_{i-1}<h<\tan\theta_i$. Let
\begin{equation}\label{vertexedge}E(h)\!:=(V_i,h),\quad
P_{E(h)}(z)\!:=cz^{m_i}.\end{equation}

We call $E(h)$ a \textit{vertex edge} and $P_{E(h)}(z)$ the
\textit{associated polynomial}. The \textit{height}, or
\textit{co-slope}, of $E(h)$ is, by definition, $h(E(h))\!:=h$.

\s

Let $\PP_{ext}(\phi,\alpha)$ denote the edges
$\{E_0,...,E_{l-1}\}$ of $\PP(\phi,\alpha)$ plus the vertex edges.

\begin{conv}For an edge $E$ of $\PP(\phi,\alpha)$, $h\!:=h(E)<\infty$, we
also write $E$ as $E(h)$.\end{conv}

Now we define$$\iota: Supp(\alpha)\rightarrow
\PP_{ext}(\phi,\alpha),\quad \bar{B}\mt \iota(\bar{B}),$$where
$\iota(\bar{B})$ is the \textit{unique} edge of height (co-slope)
$h(\iota(\bar{B}))=h(B)$. \textit{This is a bijection.}

\s

Take $\bar{B}\in Supp(\alpha)$, $h\!:=h(B)$. Let $c$ be the
largest constant such that $\PP(\phi,\alpha)$ is bounded below by
the line $\mathcal{L}(h):u+v/h=c$. Let $i$ be the smallest integer
such that $h\leq \tan\theta_{E_i}$. If $h=\tan\theta_{E_i}$, all
dots on $E_i$ lie on $\mathcal{L}(h)$. If $h<\tan\theta_{E_i}$,
$V_i$ is the \textit{only} vertex lying on $\mathcal{L}(h)$.
\textit{In either case}, we define $\iota(\bar{B})\!:=E(h)$. See
Fig.\ref{fig:np4}, \S \ref{treemodel}.

\s

The corresponding associated polynomials differ merely by a
translation:
\begin{equation}\label{PB}P_B(z)=P_{\iota(\bar{B})}(z-a), \; \text{
$a$ the $\bar{B}$-coordinate of $\alpha$}.\end{equation}

\section{Proof Of The Morse Stability Theorem over
$\F$}\label{ProofMorse}

Give $B$, $h(B)<\infty$. Take $\zeta\in Z(\phi)$,
$\zeta\perp\bar{B}$. Let $\zeta_t$ be the deformation of $\zeta$
in Corollary\,\ref{CorFunLem}. Define $B_t$ to be the
\textit{unique} bar such that $h(B_t)=h(B)$,
$\zeta_t\perp\bar{B}_t$.

If $\zeta'\in Z(\phi)$ and $\zeta'\perp\bar{B}$, then
$O_y(\zeta_t-\zeta_t')=O_y(\zeta-\zeta')\geq h(B)$. Hence
$Supp(\zeta_t)$, $Supp(\zeta'_t)$ have the same set of bars of
height $\leq h(B)$.

It follows that $B_t$ is well-defined (\textit{independent} of the
choice of $\zeta$),
$$\mathcal{D}_{bar}:Bsp(\M_{1,\phi})\times
I_\F\arr Bsp(\M_1\times_\Phi I_\F), \;\, (B,t)\mt (B_t,t),$$is a
homeomorphism, $h(B_t)=h(B)$.

\begin{lem} The family $\{P_{B_t}(z)\}$ is Morse
stable (in the sense of Definition \ref{DefinitionMorseStable}).
\end{lem}
\begin{proof}In $\PP(\phi,\zeta)$, the left vertex of $E_{top}$ is 
$(m_l,q_l)$,
$m_l\geq 1$, where $m_l$ is the multiplicity of $\zeta$ (as a root
of $\phi$). Therefore $P_{top}(0)=0$, $\deg P_{top}(z)\geq 2$,
$P_{top}(z)$ is \textit{not} a monomial.

\s

Hence there exists $c$, $P_{top}'(c)=0\ne P_{top}(c)$. By Theorem
\ref{KLKP}, there exists $\gamma_\phi \in \textit{C}(val_{\phi})$,
whose $\bar{B}$-coordinate is $c$, $B\!:=\iota^{-1}(E_{top})$.
Then $c$ has a deformation $c_t$, $c_0=c$,
$$\tau_t(\gamma_\phi)(y)=\zeta_B(y)+c_ty^{h(B)},\quad \PP(\phi,
\gamma_\phi)=\PP(\phi_t,\tau_t(\gamma_\phi)).$$

\s

Let us compare the edge $E_i(\zeta_t)$ of $\PP(\phi_t,\zeta_t)$
with $E_i(\tau_t(\gamma_\phi))$ of
$\PP(\phi_t,\tau_t(\gamma_\phi))$.

\s

If $i\leq l-2$, then we clearly have $E_i(\zeta_t)\equiv
E_i(\tau_t(\gamma_\phi))$. Hence $P_{E_i(\zeta_t)}(z)=
P_{E_i(\tau_t(\gamma_\phi))}(z)$.

By the Fundamental Lemma, $\{P_{E_i(\tau_t(\gamma_\phi))}(z)\}$ is
Morse stable. Hence if $\bar{B}'\in Supp(\zeta)$ and
$h(B')<\tan\theta_{top}$, then $\{P_{B_t'}(z)\}$ is Morse stable.

\s

As for the top edges $E_{top}(\zeta_t)$,
$E_{top}(\tau_t(\gamma_\phi))$, their associated polynomials
differ merely by a translation $z\mt z+c_t-a_t$, where $a_t$ is
the $\bar{B}$-coordinate of $\zeta_t$. The stability of the latter
implies that of the former.

\s

Finally, if $h(B'')>\tan\theta_{top}$, then $P_{B''_t}(z)$ is a
monomial, hence Morse stable.\end{proof}

\s

Take $B\in Bsp(\M_{1,\phi})$, and deformation $B_t$. Recall that
$\bar{B}=\bar{B}_t=\C$. Applying the classical Morse Stability
Theorem (\S\ref{ClassicalMorse}) to $\{P_{B_t}\}$, we have
homeomorphisms $D_t$, $d_t$ such that

\[
\begin{array}{cccc}

\bar{B} & \stackrel{P_B} {\4\longrightarrow} & \C

\\

\Big\downarrow\vcenter{\rlap{$D_t$}} & & \Big\downarrow\vcenter{%
\rlap{$d_t$}}

\\

\bar{B}_t & \stackrel{P_{B_t}}{\4\longrightarrow}&\C

\end{array}
\]
is commutative, where $D_t$ preserves the critical points and
zeros ($d_t(0)=0$).

Given $\alpha_\phi\in B$, with $\bar{B}$-coordinate $a$. Take
$\alpha_{\phi,t}\in B_t$ whose $\bar{B}_t$-coordinate is
$D_t(a)$:$$\alpha_\phi(y)=\zeta_B(y)+ay^{h(B)},\;
\alpha_{\phi,t}(y)=\zeta_{B_t}(y)+D_t(a)y^{h(B)}.$$

Thus, we have a homeomorphism:
$$\mcd_\Phi:\M_{1,\phi}\times I_\F\rightarrow \M_1\times _\Phi
I_\F,\quad (\alpha_\phi,t)\mt (\alpha_{\phi,t},t).$$

\s

Next, take $(u,L)\in \mcv$, $B\in Bsp(\M_{1,\phi})$,
$L=L(B)\,(<\infty)$. We \textit{define} $$\mcd_\mcv((u,L),B,t))
\!:=((d_t(u),L),B_t,t).$$If $h(B)=\infty$, then $B=\{\zeta\}$,
$\zeta\in Z(\phi)$. We \textit{define}
$$\mcd_\mcv(0_\mcv,\{\zeta\},t)\!:=(0_\mcv,\{\zeta_t\},t).$$

We then have $\mcd_\mcv\circ Val_\phi=Val_\Phi\circ \mcd_\Phi$.

\m

It remains to show that $\mcd_\Phi$, $\mcd_\mcv$ preserve the
structures.

\s

Take $\mu_\phi$. Take $\zeta\in Z(\phi)$ such that
$h(\mu_\phi)=O(\mu_\phi-\zeta)$. Then
$$h(\mu_\phi)=O(\mu_\phi-\zeta)=O(\mu_{\phi,t}-\zeta_t)=h(\mu_{\phi,t}).$$

Next, $\chi_{puis}(\mu_\phi)$ and $\chi_{puis}(\mu_{\phi,t})$ can
be expressed in terms of the co-slopes of the edges of
$\PP(\phi,\zeta)=\PP(\phi_t,\zeta_t)$. Hence
$\chi_{puis}(\mu_\phi)=\chi_{puis}(\mu_{\phi,t})$.

\s

As for the contact order, first suppose $\mu_\phi\sim_{bar}
\nu_\phi$. Then there exists $\zeta\in Z(\phi)$,
$$h(\mu_\phi)=h(\nu_\phi)=O(\mu_\phi-\zeta)=O(\nu_\phi-\zeta),$$which remain
valid when the parameter $t$
is added. Hence
$\CTO(\mu_\phi,\nu_\phi)=\CTO(\mu_{\phi,t},\nu_{\phi,t})$.

Now suppose $B(\mu_\phi)\ne B(\nu_\phi)$. Take $\zeta\perp
\bar{B}(\mu_\phi)$, $\zeta'\perp\bar{B}(\nu_\phi)$. Then
$$\CTO(\mu_\phi,\nu_\phi)=O(\zeta-\zeta')=O(\zeta_t-\zeta'_t)=
\CTO(\mu_{\phi,t},
\nu_{\phi,t}).$$

\section{The Trivialization Vector Field}
\label{trivializationvectorfield}

When a coordinate system $(z_1,...,z_n)$ of $\C^n$ is chosen, we
use $\{\frac{\pt}{\pt z_1},..., \frac{\pt}{\pt z_n}\}$ to denote
the standard orthonormal basis, with hermitian product
$$<\sum a_i\frac{\pt}{\pt z_i},\sum b_i\frac{\pt }{\pt z_i}>=\sum
a_i\bar{b}_i\quad (\bar{b}_i\;\text{the complex conjugate
of}\;b_i).$$

For a holomorphic function $h(z_1,...,z_n)$, the \textit{gradient}
of $h$ (\cite{milnor}, p.33) is
$$Grad\,h\!:=\sum \overline{\frac{\pt h}{\pt z_i}}\frac{\pt}{\pt
z_i}.$$

Let $f(x,y)$, $F(x,y,t)$, $\phi_t$, $\Phi$, be as in
\S\ref{GoeThms}. To prove the Equi-singular Deformation Theorem,
we use a vector field $\vec{\CF}(x,y,t)$ which is defined in two
steps, following
 Ehresmann's idea (\cite{milnor1}), where $(x,y,t)\in
\mathcal{N}$, $\mathcal{N}$ \textit{a sufficiently small
neighborhood of} $\{0\}\times I_\C$ in $\C^2\times I_\C$.

\m

\textit{Step} \textit{One}. Take $\gamma_\phi\in
\textit{C}(val_{\phi})$, with deformation
$\gamma(y,t)\!:=\tau_t(\gamma_{\phi})(y)$ as in
(\ref{MorseStableDefi}). \textit{In this step we assume}
$\gamma(y,t)$ \textit{is a holomorphic function in} $(y,t)$,
$\gamma(0,t)\equiv 0$.

The curve-germ defined by $x=\gamma(y,t)$ is \textit{smooth},
$t\in I_\C$.

We define $\vec{\CF}_{\gamma}(x,y,t)$ as follows. The coordinate
transformation
$$\mcd:(x,y,t)\mt (x_\gamma,y_\gamma,t_\gamma)\!:=(x-\gamma(y,t),y,t)$$ is
holomorphic, $\mcd^{-1}$ transforms $F(x,y,t)$ to
\begin{equation}\label{Fgamma}F^{(\gamma)}(x_\gamma,y_\gamma,t_\gamma)\!:=
F(x_\gamma+\gamma(y_\gamma,t_\gamma),y_\gamma,t_\gamma).\end{equation}

\begin{conv}\label{Fgamma}We shall use ${F}_{x_{\gamma}}$,
${F}_{y_{\gamma}}$, $F_{t_{\gamma}}$ to denote the partial
derivatives of $F^{(\gamma)}$. The notations are simpler, but
cause no confusion.\end{conv}

Now, consider the vector field
$$\vec{V}(x_{\gamma},y_{\gamma},t_{\gamma})\!:=-A(x_{\gamma},y_{\gamma},
t_{\gamma})
x_{\gamma} \frac{\partial}{\partial
x_{\gamma}}-B(x_{\gamma},y_{\gamma},t_{\gamma})y_{\gamma}
\frac{\partial}{\partial
y_{\gamma}}+\frac{\partial}{\partial t_{\gamma}},$$where
\begin{equation}\label{AB}A\!:=\frac{\bar{x}_{\gamma}\bar{F}_{x_{\gamma}}
F_{t_{\gamma}}}
{|x_{\gamma}F_{x_{\gamma}}|^2+ |y_{\gamma} F_{y_{\gamma}}|^2},
\quad B\!:=\frac{\bar{y}_{\gamma}\,\bar{F}_{y_{\gamma}}
F_{t_{\gamma}}}{|x_{\gamma} F_{x_{\gamma}}|^2+ |y_{\gamma}
F_{y_{\gamma}}|^2},
\end{equation}
and, by the Chain Rule,
\begin{equation}\label{ChainRule}\frac{\pt}{\pt
x_{\gamma}}=\frac{\pt}{\pt x},\quad \frac{\pt}{\pt
y_{\gamma}}=\frac{\pt}{\pt y}+\frac{\pt \gamma}{\pt
y}\frac{\pt}{\pt x},\quad \frac{\pt}{\pt
t_{\gamma}}=\frac{\pt}{\pt t}+\frac{\pt \gamma}{\pt
t}\frac{\pt}{\pt x}.\end{equation}

The coefficients $A$, $B$ are chosen so that $<\!\vec{V},Grad\,
F\!>=0$. Hence $\vec{V}$ is \textit{tangent} to the level surfaces
$F=const$. (Here $Grad\,F\!:=\bar{F}_{x_{\gamma}}\frac{\pt}{\pt
x_{\gamma}}+\bar{F}_{y_{\gamma}}\frac{\pt}{\pt
y_{\gamma}}+\bar{F}_{t_{\gamma}}\frac{\pt}{\pt t_{\gamma}}$.)

We have not defined $\vec{V}$ when
$x_{\gamma}F_{x_{\gamma}}=y_{\gamma}F_{y_{\gamma}}=0$.
\textit{This we shall do in} \S \ref{proofEquiSingThm}.

Using (\ref{ChainRule}) we can express $\vec{V}$ as a vector field
in the $(x,y,t)$-space:
\begin{equation}\label{vecFsubg}\vec{\CF}_{\gamma}(x,y,t)\!:=(d\mcd)^{-1}
(\vec{V}),\end{equation}which
is tangent to $F=const$. Each trajectory (integral curve) lies on
a single level surface.

\m

When $x_{\gamma}=0$, the $\partial/{\partial x_{\gamma}}$
component of $\vec{V}$ vanishes, hence the flow generated by
$\vec{V}$ carries the $y_{\gamma}$-axis to itself (but not
necessarily point-wise fixed). \textit{The flow generated by}
$\vec{\CF}_{\gamma}$, \textit{in the} $(x,y,t)$-\textit{space},
\textit{carries the curve-germ} $\pi_*(\gamma(y,0))$ 
\textit{to}
$\pi_*(\gamma(y,t))$ at time $t$.

\m

\textit{Step Two.} \textit{We are to define} 
$\vec{\CF}(x,y,t)$.
Take $\gamma_{j,\phi}\in \textit{C}(val_{\phi})$ in
(\ref{listingC}). Take the deformation $\tau_t(\gamma_{j,\phi})$
in (\ref{MorseStableDefi}). Write the canonical coordinate simply
as
\begin{equation}\label{gammat}\gamma_j(y,t)\!:=\tau_t(\gamma_{j,\phi})(y),
\quad 1\leq j\leq p.\end{equation}

Take an integer $N$ divisible by every
$m_{puis}(\gamma_{j,\phi})$, for instance,
$N\!:=\prod_jm_{puis}(\gamma_{j,\phi})$.

Consider the substitution map
\begin{equation}\label{substitutionmap}S_{XY}:(X,Y,T)\mt (x,y,t)\!:=(X,Y^N,T),
\end{equation}
and the coordinate
transformation
$$X_j\!:=X-\Gamma_j(Y,T),\quad Y_j\!:=Y,\quad T_j\!:=T,$$
where
$\Gamma_j(Y,T)\!:=\gamma_j(Y^N,T)$ is holomorphic. Like
(\ref{Fgamma}), we write
$$F^{(j)}(X_j,Y_j,T_j)\!:=F(X_j+\Gamma_j,Y_j^N,T_j),\quad
1\leq j\leq p.$$

Of course $X_j=0$ is \textit{smooth}, and is mapped by $S_{XY}$ to
$\pi_*(\gamma_{j,\phi})$, $1\leq j\leq p$. (The latter may not be
mutually distinct:\;if $\gamma_{1,\phi}(y)$, $\gamma_{2,\phi}(y)$
are conjugates, then
$\pi_*(\gamma_{1,\phi})=\pi_*(\gamma_{2,\phi})$.)

Therefore, for each $j$, $\vec{\CF}_{\Gamma_j}(X,Y,T)$ is defined
as in (\ref{vecFsubg}), being the vector field
$$\vec{V}_j(X_j,Y_j,T_j)\!:=-A_j(X_j,Y_j,T_j)X_j\frac{\pt}{\pt
X_j}-B_j(X_j,Y_j,T_j)Y_j \frac{\pt}{\pt Y_j}+\frac{\pt}{\pt
T_j}$$expressed in terms of $(X,Y,T)$, where, as in (\ref{AB}),
with Convention \ref{Fgamma},$$A_j
:=\frac{\bar{X}_j\bar{F}_{X_j}F_{T_j}} {|X_jF_{X_j}|^2+ |Y_j
F_{Y_j}|^2}, \quad B_j:=\frac{\bar{Y}_j\bar{F}_{Y_j}F_{T_j}}
{|X_jF_{X_j}|^2+ |Y_j F_{Y_j}|^2}.$$

\m

Let us write $\hat{X}_k\!:=X_1\cdots X_{k-1}\cdot
 X_{k+1}\cdots X_p$, and, for $(X,Y)\ne (0,0)$, define
\begin{equation}\MP_k \!:=
\frac{|\hat{X}_k|^2}{|\hat{X}_1|^2+\cdots +|\hat{X}_{p}|^2},\qquad
1\leq k\leq p.\end{equation}We call $\{\MP_k\}$ a
\textit{partition of unity}, for we have
\begin{equation}\label{partionP}\sum \MP_k=1,\quad \MP_k=1\;
\text{when} \;X_k=0,\quad \MP_k=0\; \text{when} \;X_j=0,\; j\ne
k.\end{equation}

The $\MP_k$'s are \textit{real analytic} at every $(X,Y) \ne
(0,0)$ in the sense of Convention\,\ref{conv1}.

Now we use the $\MP_k$'s to \textit{``patch up"} the vectors
$\vec{\CF}_{\Gamma_j}$:
\begin{equation}\label{vectorv}\vec{v}(X,Y,T)\!:=
\MP_1\vec{V}_1+\cdots +\MP_p\vec{V}_p\quad(\text{all expressed in
$X,Y,T$}),\end{equation}and, using the differential $d S_{XY}$ of
the substitution map $S_{XY}$, define
\begin{equation}\label{vectorF}\vec{\CF}(x,y,t)\!:=dS_{XY}(\vec{v}).
\end{equation}

We must show $\vec{\CF}$ is well-defined, since $S_{XY}$ is a
many-to-one mapping.

Let $\theta\!:=e^{{2\pi \ii}/N}$. A conjugation $y^{1/N}\mt
\theta^j y^{1/N}$ permutes the $\gamma_k$'s, the $X_k$'s and the
$\MP_k$'s. Hence $\vec{v}(X,Y,T)$ is 
\textit{invariant under these
conjugations}. It follows that $\vec{\CF}$ is well-defined.

\textit{We shall show, in \S 
\ref{proofEquiSingThm}},
that
$\vec{\CF}$ can be extended continuously throughout $\mathcal{N}$,
so that $\vec{\CF}(x,y,t)$ is well-defined in $\mathcal{N}$,
tangent to $F=const$. Because of (\ref{partionP}), \textit{the
flow generated by $\vec{\CF}$ carries the curve-germ
$\pi_*(\gamma_j(y,0))$ to $\pi_*(\gamma_j(y,t))$} \textit{at time}
$t$, $1\leq j\leq p$.

\begin{attention}\label{attentionPham} \textit{It is important to
point out
what the above does not say}.

Take $\gamma_{j,\phi}$. Of course there exists
$\gamma_j(y)\!:=\gamma_{j,\phi}(y)+\cdots$ such that
$F_x(\gamma_j(y),y,0)=0$. Hence $F_x(x,y,0)=0$ on the curve-germ
$\Delta\!:=\pi_*(\gamma_j(y))$; $\Delta$ is called a ``polar" of
$F(x,y,0)$.

Note that
$\mathcal{C}_{ord}(\Delta,\pi_*(\gamma_{j,\phi}))>O_y(\gamma_{j,\phi})$,
and, in general, $\Delta\ne \pi_*(\gamma_{j,\phi})$. Following the
flow, $\Delta$ reaches $\Delta_t$ at time $t$. \textit{The above
does not say $\Delta_t$ is necessarily a polar of
$F(x,y,t)$}.

For example, the Pham family $P(x,y,t)$ in (\ref{Phamfamily}) has
only \textit{one} polar when $t=0$, but \textit{two} polars when
$t\ne 0$. We have no idea whether the polar at $t=0$ will flow to
one of the two polars, or more likely to neither. In the
\textit{blurred space} $\C_{*/P_t}$, however, there is a
\textit{unique critical point} for each $t\in I_\C$; they
constitute a single orbit of $\{\eta_t\}$.

A critical point is an equivalence class in $\C_*$ containing at
least one polar.  \textit{We see the critical point at all
time
$t$, but cannot keep track of the polars.}
(This is like the
Arakawa in Japan, a river which flows by Saitama University. We
see the river bed, but cannot predict the position of the flow,
whence, literally, the name ``Arakawa"--Wild River.)
\end{attention}

\section{Proof of the Equi-singular Deformation
Theorem}\label{proofEquiSingThm}

Recall that $F$ is mini-regular in $x$. In a sector $|y|\leq
\epsilon|x|$, the behavior of $F$ is dominated by $x^m$,
$m\!:=O(F)$. Hence there is nothing to worry about in this sector.

We shall henceforth restrict our attention to a sector $|x|<
K|y|$, $K$ sufficiently large.

\begin{notation}In this section we write $g\lesssim h$ if $g\leq
Ch$, $C>0$ a constant; $g\approx h$ means $g\lesssim h\lesssim g$;
and $g\ll h$ means $g/h\rightarrow 0$.\end{notation}

Next we show how $\vec{v}$ in (\ref{vectorv}) and $\vec{\CF}$ in
(\ref{vectorF}) generate homeomorphisms.

\s

 The following is a parameterized version of the Proposition in
\cite{laurent}, p.347.

\begin{lem}\label{lemlesssim}For
$F^{(j)}(X_j,Y_j,T_j)\!:=F(X_j+\Gamma_j(Y_j,T_j),Y^N_j,T_j)$, we
have
\begin{equation}\label{lesssim}|F_{T_j}|\lesssim
|X_jF_{X_j}|+|Y_jF_{Y_j}|,\quad (X_j,Y_j,T_j)\in
S_{XY}^{-1}(\mathcal{N}),
\end{equation}where $
1\leq j\leq p$, $\mathcal{N}$ as in
\S\ref{trivializationvectorfield}. (We use
Convention\,(\ref{Fgamma}): $F_{X_j}\!:=F^{(j)}_{X_j}$, $etc.$.)

We can extend $A_j$, $B_j$ real analytically to
$S_{XY}^{-1}(\mathcal{N})- \{0\}\times I_\C$, where they are
bounded.

Define $\vec{\CF}(0,0,t)\!:=\frac{\pt}{\pt t}$. Then $\vec{\CF}$
is continuous on $\mathcal{N}$,
\begin{equation}\label{w-equation}\|\vec{\CF}(x,y,t)
-\partial/{\partial t}\|\lesssim|x|+|y|,\quad (x,y,t)\in
\mathcal{N},\end{equation}and $\vec{\CF}$ is real analytic in
$\mathcal{N}-\{0\}\times I_\C$.

It follows that a trajectory of $\vec{\CF}$, with initial point
outside $I_\C$, will never reach $I_\C$.

The flow generated by $\vec{\CF}$ carries $\pi_*(\gamma_{j,\phi})$
to $\pi_*(\tau_t(\gamma_{j,\phi}))$, $1\leq j\leq p$.
\end{lem}
\begin{proof}We use the Curve Selection Lemma to prove
(\ref{lesssim}). Let $\rho(s)$ be a given analytic arc. It
suffices to show that (\ref{lesssim}) holds along $\rho$. We can
assume $\rho(0)=0$.

Take $\varepsilon'>0$, sufficiently small. For each pair
$\Gamma_k, \Gamma_s$ (in \S\ref{trivializationvectorfield}), let $
d_{ks}\!:=O_Y(\Gamma_k-\Gamma_s)$.

Define a \textit{horn} \textit{neighborhood} of $\Gamma_k$ of
order $d_{ks}$ by
\begin{equation}\label{varepsilonprime}H_{d_{ks}}(\Gamma_k)\!:=\{(X,Y,T)\,|\,
|X-\Gamma_k(Y,T)|<
\varepsilon'|Y|^{d_{ks}}\}.\end{equation}This
 is a sub-analytic set, hence either $Im(\rho)-\{0\}\subset
H_{d_{ks}}(\Gamma_k)$ or $Im(\rho)\cap
H_{d_{ks}}(\Gamma_k)=\emptyset$.

\s

Let us first consider the case where $Im(\rho)-\{0\}$ is contained
in at least one of the horn neighborhoods. In this case, by
permuting the indices, if necessary, we can assume
$H_{d_{12}}(\Gamma_1)$ is the \textit{smallest} horn neighborhood
containing $Im(\rho)-\{0\}$.

\s

We then work in the coordinate system $(X_1,Y_1,T_1)$, writing
$\rho(s)=(X_1(s), Y_1(s),T_1(s))$.

\s

Let $\rho_\pi(x)\!:=(X_1(s),Y_1(s))$. We show (\ref{lesssim})
holds on the surface $Im(\rho_\pi)\times I_\C$.

Let us write
\begin{equation}\label{Fsub0}
F^{(1)}(X_1,Y_1,T_1)\!:=F_0^{(1)}(X_1,Y_1)+P(X_1,Y_1,T_1),\quad
 P(X_1,Y_1,0)\equiv 0.\end{equation}

By the Fundamental Lemma, $\PP(F^{(1)},0)=\PP(F^{(1)}_0,0)$, all
dots of $P(X_1,Y_1,T_1)$ lie on or above this polygon. (Newton
Polygon at $0$ is Newton Polygon in the usual sense.)

\begin{conv}Consider the vertex $V_l=(m_l,q_l)$. Suppose $m_l=0$ and
$h>\tan\theta_{top}$. In this section, we call the pair
$E(h)\!:=(V_l,h)$ \textit{also a vertex edge}, with co-slope $h$.
\end{conv}

\textit{First, suppose $Im(\rho)$ is not contained in
$X_1=0$}. We
write the coordinate of $\rho_{\pi,\C}$ as
\begin{equation}\label{rhoexpression}\rho_\pi(Y_1)\!:=rY_1^h+\cdots,
\quad r\ne 0,\quad d_{12}\leq h<\infty.\end{equation}And let
$E\!:=E(h)$ denote the \textit{unique} (possibly vertex) edge with
co-slope $h$.

\s

We first prove $(\ref{lesssim})$ for $j=1$.

\s

Assume $E\!:=E(h)$ is a proper edge ($i.e$., not a vertex edge).

In the first place we must have $P_E'(r)\ne 0$. For if
$P_E'(r)=0$, then, by Theorem \ref{KLKP}, there would exist
$\Gamma_j$, $j\geq 2$, of the form $\Gamma_j=rY_1^h+\cdots$. Then
$H_{d_{12}}(\Gamma_1)$ would not be the smallest horn
neighbourhood containing $Im(\rho)$, a contradiction.

\s

Now we collect the monomial terms of $F^{(1)}$ along $E$:
\begin{equation}\label{weightedForm}W(X_1,Y_1,T_1)\!:=
\sum a_{pq}(t)X_1^pY_1^q,
\quad (p,q)\in
E,\end{equation}which is a weighted form, $W(z,1,t)=P_E(z)$.

\s

Take $u\ne 0$. Let $t$ be fixed. By Euler's Theorem, $X_1-uY_1^h$
is a common factor of $X_1W_{X_1}$ and $Y_1W_{Y_1}$ \textit{iff}
$(X_1-uY_1^h)^2$ divides $W(X_1,Y_1,T_1)$, $i.e.$,
$P_E(u)=P_E'(u)=0$.

\m

Hence, if $(P_E(r),P_E'(r))\ne (0,0)$, then, along $\rho_\pi$,
$$O_{Y_1}(|X_1W_{X_1}|+|Y_1W_{Y_1}|)
=O_{Y_1}(|X_1F_{X_1}|+|Y_1F_{Y_1}|)=q_{_E}+m_{_E}h,$$where
$(m_{_E},q_{_E})$ is any dot on $E$. It follows that
\begin{equation}\label{LojAlongRho}L_{|X_1F_{X_1}|+|Y_1F_{Y_1}|}(\rho_\pi)=
q_{_E}+m_{_E}h.\end{equation}

All dots of $F_{T_1}$ lie on or above the line $\mathcal{L}(E)$,
$O_{Y_1}(F_{T_1})\geq q_{_E}+m_{_E}h$, proving (\ref{lesssim}).

\m

If $E=(V_k,h)$ is a vertex edge, then $|Y_1F_{Y_1}|\approx
|r|^{m_k}|Y_1|^{q_k+m_kh}\gtrsim |F_{T_1}|$. Again (\ref{lesssim})
holds.

\m

Next we prove (\ref{lesssim}) for the case $j\geq 2$. The
coordinate systems are related by
\begin{equation}\label{coordinaterelated}X_j=X_1-\delta_j(Y_1,T_1),
\quad Y_j=Y_1, \quad
T_j=T_1,\end{equation}where $\delta_j\!:=\Gamma_j-\Gamma_1$
($\delta_1\equiv 0$). Let us write
\begin{equation}\label{deltacoeff} \delta_j(Y_1,T_1)\!:=
c_j(T_1)Y_1^{O(\delta_j)}+\cdots,\quad c_j(0)\ne 0,\quad 2\leq
j\leq p.\end{equation}

Note that $$|Y_1\frac{\pt \delta_j}{\pt Y_1}|\approx|\delta_j|,
\quad |\frac{\pt \delta_j}{\pt T_1}|\lesssim|\delta_j|,$$ and, by
the Chain Rule,
\begin{equation}\label{related}X_jF_{X_j}=(X_1-\delta_j)F_{X_1},\; Y_jF_{Y_j}
=Y_1F_{Y_1}+Y_1
\frac{\pt \delta_j}{\pt Y_1}F_{X_1},\;F_{T_j}=F_{T_1}+\frac{\pt
\delta_j}{\pt T_1}F_{X_1}.\end{equation}

Now, if $h\geq O(\delta_j)$, then, along $\rho_\pi$,
$|X_1-\delta_j|\approx|\delta_j|$. Hence (\ref{lesssim}) follows
from
$$|X_jF_{X_j}|+|Y_jF_{Y_j}|\approx|\delta_j F_{X_1}|+|Y_1F_{Y_1}|,\quad
|F_{T_j}|\lesssim|F_{T_1}|+|\delta_j F_{X_1}|,$$and
(\ref{lesssim}) when $j=1$.

\s

Suppose $h<O(\delta_j)$. Then $|X_1-\delta_j|\approx|X_1|$ along
$\rho_\pi$. Again we have (\ref{lesssim}).

\m

\textit{Suppose $Im(\rho)$ is contained in $X_1=0$}. 
In this
case
 $F$ is divisible by $X_1^{m_l}$, $F_{T_1}$ is divisible by
$X_1^{k}$, $k\geq m_l$. Again, we have (\ref{lesssim}).

\m

It remains to consider the case where no horn neighborhood
contains $Im(\rho)-\{0\}$. We again write $\rho_\pi$ as
(\ref{rhoexpression}) where now $h<d_{jk}$ $\forall$ $j,k$. The
same argument proves (\ref{lesssim}).

\m

Next we show how $A_j$, $B_j$ can be extended across
$S_{XY}^{-1}(\mathcal{N})-\{0\}\times I_\C$.
\begin{lem}\label{denominatorofAB} The zero set of the denominator of $A_j$,
$B_j$
is
$$\{(X,Y,T)|X_jF_{X_j}=Y_jF_{Y_j}=0\}=\begin{cases}\bigcup
\{X_k=0\,|\,X_k^2\;\text{divides}\;F\},\\\{0\}\;\;\;\text{if no
such $X_k$ exists}.\end{cases}$$
\end{lem}
\begin{proof}Recall that we work in the sector $|x|<K|y|$, hence $Y\ne 0$.

Suppose $X_j=F_{Y_j}=0$. In $\PP(F^{(j)},\Gamma_j)$, which is
independent of $T_j$, we must have $m_l\geq 2$. Hence $F$ is
divisible by $X_j^2$.

Let $\rho(s)$ be an analytic arc along which $F_{X_j}=F_{Y_j}=0$.
Then $F_X=F_Y=0$ along $\rho$. As before, we choose $X_1$, define
$\rho_\pi$, $F_0(X_1,Y_1)$, $etc.$, as in (\ref{Fsub0}), and then
$$\frac{\pt F^{(1)}_0}{\pt X_1}=\frac{\pt
F^{(1)}_0}{\pt Y_1}=0\quad \text{along} \quad \rho_\pi.$$ Hence
$Im(\rho_{\pi,\C})$ must be $X_1=0$ ($T_1=0$), $X_1^2$ divides
$F_0$, hence also $F$.\end{proof}

\s

We can now complete the proof of Lemma \ref{lemlesssim}.

The real meromorphic functions $A_j$, $B_j$ are \textit{bounded},
by (\ref{lesssim}). Hence, if $n$ is the largest integer such that
$X^n_k$ divides both $X_jF_{X_j}$ and $Y_jF_{Y_j}$, then $X_k^n$
must also divides $F_{T_j}$.

It follows that $A_j$, $B_j$ are \textit{defined and real
analytic} on $S_{XY}^{-1}(\mathcal{N})-\{0\}\times I_\C$.

\m

Hence $\vec{\CF}(x,y,t)$, with $\vec{\CF}(0,0,t)\!:=\frac{\pt}{\pt
t}$, is \textit{continuous} in $\mathcal{N}$, 
\textit{satisfying}
(\ref{w-equation}), and is \textit{real analytic} in
$\mathcal{N}-\{0\}\times I_\C$.

\s

Finally, $\mathcal{P}_kX_k=0$ along every $X_j=0$. Hence
$\vec{\CF}$ carries $\pi_*({\gamma_{j,\phi}})$ to
$\pi_*(\tau_t(\gamma_{i,\phi}))$.
\end{proof}

Now, using a well-known argument (\cite{kuo}), (\ref{w-equation})
implies that $\vec{\CF}$ generates a homeomorphism $H$ having
properties (\ref{homeoh}), (1) and (2) in the Equi-singular
Deformation Theorem.

\m

Next we prove $(3)$. Let $\rho(s)=(x(s),y(s))$ be a given analytic
arc in the $(x,y)$-plane.

Consider $\rho^{(N)}(s)\!:=(x(s^N),y(s^N))$. This arc can be
lifted by $S_{XY}$ to an analytic arc
$$\rho_{_{XY}}(s)\!:=(X(s),Y(s)),\quad S_{XY}\circ
\rho_{_{XY}}=\rho^{(N)},$$where $Y(s)$ is an \textit{integral}
power series in $s$ obtained by solving $Y^N=y(s^N)$. (Of course,
there are $N$ liftings of $\rho^{(N)}$ to the $(X,Y)$-space; in
general, however, we cannot lift $\rho$.)

\s

Let us write $\Gamma_j\!:=\Gamma_j(Y,0)$. By permuting the
indices, if necessary, we can assume
\begin{equation}\label{hsmallest}
\begin{split}h
&\!:=\MC_{ord}(Im(\rho_{_{XY,\C}}),Im(\Gamma_1))=\cdots
=\MC_{ord}(Im(\rho_{_{XY,\C}}),Im(\Gamma_r))\\
&>\MC_{ord}(Im(\rho_{_{XY,\C}}),Im(\Gamma_{r+i})),\quad 0<i\leq
p-r,\end{split}\end{equation}where $1\leq r\leq p$,
$h/N=\CTO(Im(\rho_{_\C}),\pi_*(\gamma_{j,\phi}))$, $1\leq j\leq
r$.

\m

\textit{There are now three cases to consider}: (a)
$h/N\leq
h(\gamma_{1,\phi})$, $h< \infty$, (b)
$h(\gamma_{1,\phi})<h/N<\infty$, and (c) $h=\infty$, that is,
$Im(\rho_{_\C})=\pi_*(\gamma_{1,\phi})$.

\s

Case (c) is easy: $Im(\rho_{_{XY,\C}})$ is the $Y_1$-axis, along
which $\vec{v}=\vec{V}_1$ is analytic, whence $(3)$.

\m

Consider case (a). Let $E\!:=E(h)$ be the unique edge in
$\PP_{ext}(F_0^{(1)},0)$ of co-slope $h$. ( Of course, $E$ can be
a vertex edge $(V_e,h)$. In this case $m_e>0$, $P_E(z)$ is a
monomial.)

\s

As before, we work in the coordinate system $(X_1,Y_1,T_1)$.

Take a coordinate of $Im(\rho_{_{XY,\C}})$,
$$\rho_{_{XY}}(Y_1)\!:=u_\rho Y_1^h+\cdots, \quad u_\rho\ne
0.$$A coordinate of $Im(\rho_{_\C})$ is
$$\rho(y)=\gamma_{1,\phi}(y)+[u_\rho y^{h/N}+\cdots].$$

\s

As before, $P_E'(u_\rho)\ne 0$. (Same argument: otherwise,
(\ref{hsmallest}) would fail.)

\s

\textit{Now, let us first assume that $h$ is an
integer}, so that
the substitution map
$$S_{uv}: (u,v,t)\mt (X_1,Y_1,T_1)\!:=(uv^h,v,t),$$is holomorphic,
where
\begin{equation}\label{chainruleuv}X_1\frac{\partial}{\partial
X_1}=u\frac{\partial}{\partial u},\quad
Y_1\frac{\partial}{\partial Y_1} =-h\, u\frac{\partial}{\partial
u}+v\frac{\partial}{\partial v},\quad \frac{\pt}{\pt
T_1}=\frac{\pt}{\pt t}.\end{equation}

We then define
$$\label{U}\vec{U}(u,v,t)\!:=(dS_{uv})^{-1}(\vec{v}).$$

\begin{question}\label{Qu}At which $(u,v,t)$ is $\vec{U}$ well-defined and 
real
analytic?\end{question}

To answer this, we need a careful analysis of the denominator of
$A_j$, $B_j$:
$$D_j\!:=|X_jF_{X_j}|^2+|Y_j F_{Y_j}|^2,\quad 1\leq j\leq p,$$and also that
of $\mathcal{P}_k$, when the substitution $S_{uv}$ is made.

\begin{lem}\label{denoofAB}Let $L(E)\!:=q_{_E}+m_{_E}h$, $(m_{_E},q_{_E})
\in E$.
Then
\begin{equation}\label{DsukandC(u,t)}D_j(uv^h,v,t)=C_j(u,t)v^{2L(E)}+
\cdots,\quad 1\leq j\leq r,\end{equation} where $C_j(u,t)$ is a
polynomial in $u$, $\bar{u}$, coefficients in $t$.

For $u\ne 0$, $C_j(u,t)=0$ iff $P_E(u)=P_E'(u)=0$. For any $u$,
$P_E'(u)\ne 0$ implies $C_j(u,t)\ne 0$.
\end{lem}
(The coefficients of $P_E(u)$ are functions of $t$. To say
$P_E(u)=P_E'(u)=0$ means that when $t$ is fixed, $u$ is a multiple
root of $P_E(z)$.)
\begin{proof}Consider the weighted form $W(X_1,Y_1,T_1)$ in
(\ref{weightedForm}). Let $$w_1(z,t)\!:=W_{X_1}(z,1,t), \quad
w_2(z,t)\!:=W_{Y_1}(z,1,t).$$

Take $j$, $1\leq j\leq r$. By (\ref{hsmallest}), $O(\delta_j)\geq
h$, $\delta_j$ being defined in (\ref{coordinaterelated}). Let us
write
$$
\hat{c}_j(t)\!:=
\begin{cases} c_j(t)& \text{if}\;\,
O(\delta_j)=h,\\
0& \text{if}\;\, O(\delta_j)>h,
\end{cases}
$$where $c_j(t)$ is defined in (\ref{deltacoeff}). We claim that
(\ref{DsukandC(u,t)}) holds if we take
$$C_j(u,t)\!:=|(u-\hat{c}_j(t))w_1(u,t)|^2+
|w_2(u,t)+h\hat{c}_j(t)w_1(u,t)|^2,$$
which
is of course a polynomial in $u,\bar{u}$. Indeed, we have, as
before,
\begin{equation}\label{ExpreessDsubj}
D_j=|(X_1-\delta_j)F_{X_1}|^2+|Y_1F_{Y_1}+
Y_1\frac{\pt \delta_j}{\pt Y_1}F_{X_1}|^2,
\end{equation}whence the
leading term of $D_j(uv^h,v,t)$ is the above $C_j(u,t)$.

The $\hat{c}_i(t)$'s are roots of $w_1(z,t)=P_E'(z)$. Hence
$C_j(u,t)=0$ \textit{iff} $w_1(u,t)=w_2(u,t)=0$.

For $u\ne 0$, $w_1=w_2=0$ \textit{iff} $(X_1-uY_1^h)^2$ divides
$W(X_1,Y_1,T_1)$, $i.e.$, $P_E(u)=P_E'(u)=0$. But we already know
$P_E'(0)=0$. Hence $P_E'(u)\ne 0$ implies $u\ne 0$, $C_j(u,t)\ne
0$.
\end{proof}

\begin{lem}For $r<j\leq p$, we have
\begin{equation}\label{Dsubj}D_j(uv^h,v,t)=(1+h^2)|c_j(t)w_1(u,t)|^2\,
v^{2[L(E)-h+O_y(\delta_j)]}+\cdots.\end{equation}Hence if
$P_E'(u)\ne 0$, then $O_v(D_j(uv^h,v,t))=2[L(E)-h+O_y(\delta_j)]$.

If $P_E(u)\ne 0$, then $O_v(D_j(uv^h,v,t))\leq 2L(E)$.
\end{lem}
\begin{proof}If $j>r$, then $h>O(\delta_j)$ and hence
$$uv^h-\delta_j=-c_j(t)v^{O(\delta_j)}+\cdots,\quad c_j(0)\ne
0.$$

Then, as can be observed from (\ref{ExpreessDsubj}), the leading
term of $D_j$ is that of (\ref{Dsubj}).

Next note that
$$
D_j\approx |\delta_jF_{X_1}|^2+|vF_{Y_1}+v\frac{\pt \delta_j}{\pt
Y_1}F_{X_1}|^2\approx |\delta_jF_{X_1}|^2+|vF_{Y_1}|^2 \geq
|vF_{Y_1}|^2,$$and that $O_v(vF_{Y_1}(uv^r,v,t))=L(E)$ if
$P_E(u)\ne 0$. This completes the proof.
\end{proof}

The Newton dots of $F_{T_j}$ lie on or above the line
$\mathcal{L}(E)$, hence the above lemmas imply that
\textit{$A_j\circ S_{uv}$, $B_j\circ S_{uv}$ are real 
analytic at
$(u,v,t)$ if} $P_E'(u)\ne 0$, $1\leq j\leq p$.
\begin{lem}The denominator $D_\MP\!:=|\hat{X}_1|^2+\cdots +|\hat{X}_p|^2$
of $\MP_k$ has the form
\begin{equation}\label{DMP}D_\MP(uv^h,v,t)=c(u,t)v^{2e}+\cdots,\quad
e\!:=(r-1)h+\Sigma_{j>r}\,O(\delta_j),\end{equation}where $c(u,t)$
is a polynomial in $u$, $\bar{u}$.

If $(P_E(u),P_E'(u))\not=(0,0)$ then  $c(u,t)\ne 0$.
\end{lem}
\begin{proof}For $1\leq j\leq r$, let us write
$$\mu_j(u,t)\!:=u-\hat{c}_j(t), \quad
\hat{\mu}_j(u,t)\!:=\mu_1\cdots\mu_{j-1}\cdot
\mu_{j+1}\cdots\mu_r.$$

We can compute the $\hat{X}_k$'s, $1\leq k\leq p$, using the
formula$$uv^h-\delta_k=
\begin{cases}(u-\hat{c}_k(t))v^h+\cdots,&\;\text{if}\;\, 1\leq k\leq r,\\
-c_k(t)v^{O(\delta_k)}+\cdots,&\;\text{if}\;\,r<k\leq p,
\end{cases}
$$where $c_k(0)\ne 0$ for $r<k\leq p$. For example,
$$\hat{X}_1 = [(-1)^{p-r}c_{r+1}(t)\cdots
c_p(t)](u-\hat{c}_2(t))\cdots(u-\hat{c}_r(t))v^e+\cdots,$$where
$e$ is defined in (\ref{DMP}). The equations for the other
$\hat{X}_k$'s are similar.

Then (\ref{DMP}) holds when we take
$$c(u,t)\!:=|c_{r+1}(t)\cdots
c_p(t)|^2\cdot \sum_{j=1}^r|\hat{\mu}_j(u,t)|^2.$$

If $c(u,t)=0$, then all $\hat{\mu}_j=0$, and there exist $i$, $j$,
$1\leq i<j\leq r$, $u=\hat{c}_i(t)=\hat{c}_j(t)$. This implies, in
particular, that $P_E'(u)=0$.

We claim this also implies $P_E(u)= 0$. Indeed,
$\hat{c}_i(t)=\hat{c}_j(t)$ implies
$O(\gamma_{i,\phi}-\gamma_{j,\phi})>h/N$. If $P_E(u)\ne 0$, then
$h/N=h(\gamma_{i,\phi})= h(\gamma_{j,\phi})$, hence
$\gamma_{i,\phi}=\gamma_{j,\phi}$, a contradiction.\end{proof}

The $\MP_k$'s are bounded. Hence $\MP_k\circ S_{uv}$ are real
analytic at $(u,t)$ if $(P_E(u),P_E'(u))\ne (0,0)$.

\s

Let $\mathcal{N}'$ be an open neighborhood of the $u$-axis such
that $(u,v)\in \mathcal{N}'$ implies $(uv^r,v)\in \mathcal{N}$.

\m

\textbf{Answer to Question \ref{Qu}}: \textit{Take
$(u_0,v_0,t_0)\in \mathcal{N}'\times I_\C$. If $(u_0,t_0)$ is not
a root of $P_E'(u)$, then $\vec{U}(u,v,t)$ is defined 
and
real
analytic in a neighborhood of $(u_0,v_0,t_0)$.}

\begin{lem}\label{Utangenttouspacewhenv=0}Given $j$, $1\leq j\leq r$. 
For $(u,t)$
near $(\hat{c}_j(t),t)$,
\begin{equation}\label{wconditionfor U}|\vec{U}(u,v,t)-\frac{\pt}{\pt
t}|\lesssim|u-\hat{c}_j(t)|+|v|.\end{equation}

 Moreover, $\vec{U}(u,0,t)$ is tangent to the $u$-coordinate
space, and
\begin{equation}\label{Uzero}|\vec{U}(u,0,t)-\frac{\pt}{\pt
t}|\lesssim|u-\hat{c}_j(t)|.\end{equation}
\end{lem}
\begin{proof}Let us first assume $j=1$, where $\hat{c}_1(t)\equiv
0$. The notations are simpler.

First, by (\ref{chainruleuv}),
$$(dS_{uv})^{-1}(\MP_1\vec{V}_1)=-\MP_1A_1u\frac{\pt}{\pt u}-\MP_1B_1
(-hu\frac{\pt}{\pt u}+v\frac{\pt}{\pt v}) +\MP_1\frac{\pt}{\pt
t},$$where $\MP_1A_1$, $\MP_1B_1$ are bounded.

\s

For $k$, $2\leq k\leq p$, the identity
$$\MP_kX_k\frac{\pt}{\pt X_k}=\{\MP_k^{1/2}\MP_1^{1/2}X_k|X_k|^{-1}\}|X_1|
\frac{\pt}{\pt X_1},$$
and the Chain Rule (\ref{chainruleuv}) give
$$(dS_{uv})^{-1}(\MP_k\vec{V}_k)=\tilde{A}u\frac{\pt}{\pt u}+\tilde{B}v
\frac{\pt}{\pt v}+\frac{\pt}{\pt
t},$$where $\tilde{A}$, $\tilde{B}$ are bounded.

Hence (\ref{wconditionfor U}) is true for $j=1$.

\s

Now assume $1<j\leq r$. This case is actually the same as the case
$j=1$. For in (\ref{hsmallest}), the roles of $\Gamma_1$ and
$\Gamma_j$ are interchangeable. When $\Gamma_j$ replaces
$\Gamma_1$, the leading coefficient $c_k(t)$ of $\delta_k$ in
(\ref{deltacoeff}) is replaced by $c_k(t)-c_j(t)$. The same
argument completes the proof of (\ref{wconditionfor U}).

\s

The $\frac{\pt}{\pt v}$ component vanishes when $v=0$. Hence
$\vec{U}(u,0,t)$ is tangent to the $u$-space.\end{proof}

\textit{We are now ready to complete the proof of
$(3)$ in the
case (a).}

Recall that $Im(\Gamma_1)$ is the vertical axis $X_1=0$, and
$\rho_{_{XY,\C}}$ is defined by $X_1=u_\rho Y_1^h+\cdots$, which
lifts to $\rho_{uv}$\,: $u=u_\rho+\cdots$.

We know $P_E'(u_\rho)\ne 0$. Hence the arc $\rho_{uv}$ lies in the
domain where $\vec{U}$ is analytic. The flow carries $\rho_{uv}$
to an analytic arc, at least for a sufficiently short time. Hence
the flow of $\vec{v}$ carries $\rho_{_{XY}}$ to an analytic arc,
at least for a short time. (Keep $|v|$ small, arcs short.)

This is actually so for \textit{all} $t\in I_\C$, not just for a
short time. Indeed, using Lemma \ref{Utangenttouspacewhenv=0}, a
well-known argument shows that a point on $\rho_{uv}$, following
the flow, will \textit{never} reach a point where $\vec{U}$ is not
analytic, whence the trajectory is defined for \textit{all} $t$.

The initial point $(u_\rho,0,0)$ of $\rho_{uv}$ \textit{remains}
in the $u$-coordinate space for all time, because $\vec{U}$ is
tangent to the $u$-space. By (\ref{Uzero}), it will never reach
$\hat{c}_j(t)$.

Hence, downstairs, the geo-arc $Im(\rho)$ is carried by the flow
of $\vec{\CF}$ real analytically, sweeping out a ``geo-arc wing".
This completes the proof of ($3$) in the case (a) when $h\in
\Z^+$.

\s

Suppose $h\!:=N_1/M_1\not \in \Z^+$. We can make a further
substitution $Y_1\mt Y_1^{M_1}$. The vector field lifts, retaining
the same form, $h$ is magnified to $N_1$. The same argument
applies.

\m

\textit{Finally, consider case} (b). This can be treated as a
special case of (a), as follows.

\s

In $\PP(\phi,\gamma_{1,\phi})$, $E_{top}$ has left vertex
$V_l=(0,q_l)$, since $h(\gamma_{1,\phi})<\infty$. We may call
$E\!:=(V_l,h)$ an ``\textit{artificial vertex edge}", of 
co-slope
$h$, with Lojasiewicz exponent $L(E)\!:=q_l$.

Let $V_l$ represent the term $cy^{q_l}$, $c\ne 0$. The
\textit{associated polynomial} is $P_{E}(z)\!:=c$.

The above argument for the case (a) can then be repeated. It is
actually easier. The lemmas remain true for the artificial vertex
edge $E$, where now $P_E(u)=c\ne 0$.

This completes the proof of (3).

\m

\textit{We then use the same idea to prove (4).}

\s

Take $\alpha_{*/f}\in \C_{*/f}$ and a canonical coordinate
$\alpha_\phi(y)$. We can assume
\begin{equation}\label{assume}h(\alpha_\phi)=O_y(\alpha_\phi-\zeta_1)\geq
O_y(\alpha_\phi-\zeta_i),\quad
O_y(\alpha_\phi-\gamma_{1,\phi})\geq
O_y(\alpha_\phi-\gamma_{j,\phi}),\end{equation} where $\zeta_i\in
Z(\phi)$, $\gamma_{j,\phi}\in \textit{C}(val_{\phi})$.

\s

We now repeat the argument in (3) to show $\eta_t(\alpha_{*/f})$
is well-defined. Let us write
\begin{equation}\label{usubzero}\alpha_\phi(y)\!:=
\zeta_1(y)+[u_0y^{h(\alpha_\phi)}+\cdots],\quad u_0\ne
0,\end{equation}where
 $\alpha_{*/f}$ is \textit{completely} determined by $u_0$. (Terms
in ``$+\cdots"$ play no role.)

\s

Here we have assumed that $h(\alpha_\phi)<\infty$. (If
$\alpha_\phi=\zeta_1$, we can apply Corollary \ref{CorFunLem}.)

\s

Take an analytic arc $\rho$ such that $\pi_*(\rho_{_\C})\in
\alpha_{*/f}$. As in (3), we lift $\rho$ to $\rho_{_{XY}}$, and
analyze how $\rho_{_{XY}}$ is being carried by $\vec{v}(X,Y,T)$ in
(\ref{vectorv}).

This time we use the coordinate system:
$$X_\zeta\!:=X-\zeta_1(Y^N, T),\quad Y_\zeta\!:=Y, \quad T_\zeta\!:=T,
\quad (y\!:= Y^N),$$where
$N$ is divisible by $m_{puis}(\alpha_\phi)$ and
$m_{puis}(\zeta_i)$.

The substitution $(X_\zeta,Y_\zeta)\!:=(uv^h, v)$, $h\!:=N
h(\alpha_\phi)$, lifts $\rho_{_{XY}}$ to an analytic arc
$\rho_{uv}$ in the $(u,v)$-space whose initial point is $(u_0,0)$.

As in (3), $\vec{v}$ is lifted to $\vec{U}$; $\rho_{uv}$ lies in
the domain where $\vec{U}$ is real analytic.

Hence, following the flow of $\vec{U}$ for time $t$, $\rho_{uv}$
reaches an analytic arc, denoted by $\rho_{uv,t}$, with initial
point $(u_t,0,t)$, where $u_t$ is real analytic in $t$.

\textit{An important observation} is that $(u_0,0)$, and hence
also $(u_t,0,t)$, depend \textit{only} on $\alpha_{*/f}$,
\textit{not} on the choice of $\rho$. It follows that
$\eta_t(\alpha_{*/f})$ is well-defined.

\s

(Note. Even if $\rho$, $\mu$ have the \textit{same}
complexification, $Im(\rho_{_\C})=Im(\mu_{_\C})$, the above
argument \textit{does not} prove that the geo-arcs
$Im(\eta_t\circ\rho)$, $Im(\eta_t\circ\mu)$ lie on a \textit{same}
curve-germ.

It \textit{merely} shows that the curve-germs containing
$Im(\eta_t\circ \rho)$ and $Im(\eta_t\circ \mu)$ are equivalent
under $\sim_{F_t}$. This is the ``Arakawa phenomenon":
\textit{only the blurred point is well-defined}.)

\s

To complete the proof of (4), we need to know how to obtain
$\PP(f,\alpha_\phi)$ from $\PP(f,\zeta_1)$.

\s

If $h(\alpha_\phi)=\infty$, the two polygons are identical.

\s

Assume $h(\alpha_\phi)<\infty$ ($u_0\ne 0$) in (\ref{usubzero}).
Let $\{(m_1,q_1), ...,(m_l,q_l)\}$ denote the vertices of
$\PP(f,\zeta_1)$, where, of course, $m_l\geq 1$.

Take $k$ ($k\leq l$) such that $\tan\theta_{k-1}\leq
h(\alpha_\phi)<\tan\theta_k$.

Let $(0,q')$ be the left vertex of the edge $E(h(\alpha_\phi))$.
(If $\tan\theta_{k-1}< h(\alpha_\phi)<\tan\theta_k$, then
$E(h(\alpha_\phi))$ is a vertex edge.) The vertices of
$\PP(f,\alpha_\phi)$ are
$$
\begin{cases}(m_1,q_1),...\,,(m_{k-1},q_{k-1}), (0,q'),\quad
&\text{if}\quad \tan\theta_{k-1}=h(\alpha_\phi),\\
(m_1,q_1),...,(m_l,q_l),(0,q'),\quad &\text{if}\quad
\tan\theta_{k-1}<h(\alpha_\phi).\end{cases}$$

\s

We can also obtain $\PP(F_t,\eta_t(\alpha_{*/f}))$ from
$\PP(F_t,\zeta_{1,t})$ in the same way, since $u_t\ne 0$.

By Corollary\,\ref{CorFunLem}, we know
$\PP(F_t,\zeta_{1,t})=\PP(f,\zeta_1)$. Hence
$\PP(F_t,\eta_t(\alpha_{*/f}))=\PP(f,\alpha_\phi)$. It follows
that $h(\eta_t(\alpha_{*/f}))$,
$\chi_{puis}(\eta_t(\alpha_{*/t}))$ are constants.

\s

It remains to consider the contact order. From what we have
proved,
$$\CTO(\alpha_{*/f},\pi_*(\zeta_i))=\CTO(\eta_t(\alpha_{*/f}),
\pi_*(\zeta_{i,t})),\quad \zeta_i\in Z(\phi).$$
And, by Corollary \ref{usubzero},
$$\mathcal{C}_{ord}(\pi_*(\zeta_{i,t}),\pi_*(\zeta_{j,t}))=\mathcal{C}_{ord}
(\pi_*(\zeta_i),\pi_*(\zeta_j)),\quad \zeta_i,\;\zeta_j\in Z(\phi).$$

The same holds for $\beta_{*/f}$. It follows that
$\CTO(\eta_t(\alpha_{*/f}),\eta_t(\beta_{*/f}))$ is independent of
$t$. This completes the proof of (4).

\s

The vector field $\vec{\CF}$ is defined in such a way that $(5)$
is true.

\s

Finally, suppose $\Phi$ is Morse stable. Then
(\ref{valuationpreserve}) is a consequence of
(\ref{MorseCriticalValue}). This completes the proof of the
Equi-singular Deformation Theorem.

\m

\textit{To prove the Truncation Theorem}, note that
$\PP(f,\alpha_\phi)$ and $\PP(u\cdot f,\alpha_\phi)$ have the
\textit{same} Newton dots, where $u(0,0)\ne 0$. (The dots in the
interior of the polygons may be different.)

The monomial terms of $\PP(u\cdot f,\alpha_\phi)$ are those of
$\PP(f,\alpha_\phi)$ multiplied by the \textit{same} constant
$u(0,0)$. Hence the critical points $\textit{C}(val_{\phi})$ is
\textit{unchanged} when $f$ is multiplied by a unit, $F_{root}$ is
\textit{obviously} Morse stable. Now apply the Equi-singular
Deformation Theorem.

\section{Appendix 1. The Stability Theorem}
\label{ClassicalMorse}

\begin{TMT}A Morse stable family $\{p_t(x)\}$, as in
Definition\,\ref{DefinitionMorseStable}, $\K=\R$, $\C$, or $\F$,
admits a real-analytic trivialization. That is to say, there exist
$t$-level preserving, real-analytic, homeomorphisms (deformations)
$\mcd$ and $d$, such that
\[
\begin{array}{cccc}
\K\times I_{\K} & \stackrel{g\;\;} {\4\longrightarrow} &\K\times
I_{\K}

\\

\Big\downarrow\vcenter{\rlap{$\mcd$}} & & \Big\downarrow\vcenter{
\rlap{$d$}}
\\

\K\times I_{\K}& \stackrel{G\;\;}{\4\longrightarrow}&\K\times
I_{\K}

\end{array}
\] is commutative, where $d(0,t)=(0,t)$, $g(x,t)\!:=(p_0(x),t)$,
$G(x,t)\!:=(p_t(x),t)$.

(In the case $\K=\C$, moduli can appear, we cannot ask $\mcd$, $d$
to be holomorphic.)
\end{TMT}

\begin{proof}The \textit{critical value set} of $G$ is, by 
definition,
$$V_{crit}(G)\!:=\{(v_t,t)\in \K\times I_\K\,|\,\exists\,c_t\in
\K,\, p_t'(c_t)=0,\,v_t=p_t(c_t)\},$$and
$\pi:V_{crit}(G)\rightarrow I_\K$ is a fibration, $\{0\}\times
I_\K$ is either disjoint from, or contained in, $V_{crit}(G)$.

\m

Let us first consider the case $\K=\R$, which exposes the main
ideas. Take a vector field
$$\vec{v}(x,t)\!:=a(x,t)\frac{\pt}{\pt x}+\frac{\pt}{\pt t},
\quad a(x,t)\;\text{analytic},$$which is defined on,
and tangent to, $V_{crit}(G)\cup(\{0\}\times I_\R)$.

Then, using Cartan's Theorem B, or the Lagrange Interpolation
Formula, we can extend $a(x,t)$ to a real analytic function
defined for $(x,t)\in \R\times I_\R$.

In other words, $\vec{v}(x,t)$ is now defined and real analytic on
$\R\times I_\R$, tangent to $V_{crit}(G)$ and $\{0\}\times I_\R$.
Integrating this vector field gives an analytic deformation
$$d:\R\times I_\R\rightarrow \R\times I_\R, \quad (x,t)\mt
(x_t,t),$$such that $d(V_{crit}(g))=V_{crit}(G)$, and $\{0\}\times
I_\R$ is fixed.

Using $d$ as an identification, we assume
$V_{crit}(g)=V_{crit}(G)$ (``straightening up" $V_{crit}(G)$).

Take $c\in \textit{C}(p_0)$, $p_0(c)\!:=v$. Then $c$ admits a
\textit{unique} \textit{continuous} deformation $c_t$ in
 $\textit{C}(p_t)$, $p_t(c_t)=v$, $c_0=c$. As $m_{crit}(c_t)$ is
constant, $c_t$ is necessarily \textit{real analytic}. Hence
\begin{equation}\label{localformofp(x)}p_t(x)-v=unit\cdot (x-c_t)^
{m_{crit}(c)+1}, \;
(x,t)\;\text{near}\;(c_t,t),
\end{equation}Then $\frac{\pt}{\pt t}$ lifts to a unique real analytic vector
field in
$\R\times I_\R$ which generates $\mathcal{D}$.

\m

For the case $\K=\C$, we still have $d$, which is (merely) real
analytic, and also (\ref{localformofp(x)}), where $c_t$ is
holomorphic in $t$. Thus $\frac{\pt}{\pt t}$ admits a
\textit{local} real analytic lifting. By Cartan's Theorem B, a
real analytic \textit{global} lifting exists. This completes the
proof.\end{proof}

\section{Appendix 2. Tree Models}\label{treemodel}

\textit{The tree-model} (\cite{kuo-lu}) is best explained by an
example. Take $\phi(\xi)=(\xi^2-y^3)^2-4\xi y^5$. The
Puiseux roots are
$$\zeta_i(y)=\pm y^{3/2}\pm y^{7/4}+\cdots, \quad 1\leq i\leq 4.$$

The tree-model is shown in Fig.\,\ref{fig:np3}. Tracing upward
from the tree root to a tip along solid line segments amounts to
identifying a Puiseux root.

There are three bars, of height $3/2$, $7/4$, $7/4$ respectively.
Each is indicated by a horizontal line segment (whence the name);
the associated polynomial has \textit{at least two} distinct
roots.

At these heights, the $\zeta_i$'s split away from each other.
\textit{Bars without this property are not
indicated}; they
correspond to the \textit{vertex edges}, Fig.\,\ref{fig:np4}.

The three critical points (polars) $\gamma_j$ are indicated by
dashed lines; their positions relative to the $\zeta_i$ are
justified by Theorem \ref{KLKP}, exposing the contact orders.

\s

The Newton Polygon $\PP(\phi,\zeta_i)$ is shown in
Fig.\,\ref{fig:np4}. The dotted segments $e_1$, $e_2$, $e_3$
indicate vertex edges, where
$$h(e_1)<3/2<h(e_2)<7/4<h(e_3)<\infty.$$

 \begin{figure}[htt]

 \begin{minipage}[b]{.4\linewidth}
   \centering
   \begin{tikzpicture}[scale=1.0]

 \foreach \x in {3}{\draw (\x, 0) -- (\x,2 ); } \draw(0,4) --
(0,4.5)node[above]{$\zeta_1$};\draw(2,4) --
(2,4.5)node[above]{$\zeta_3$};

\draw(4,4) -- (4,4.5)node[above]{$\zeta_2$};

\draw(6,4) -- (6,4.5)node[above]{$\zeta_4$};

\draw (0,4) -- (2,4)(2,4) node[right] {$\frac{7}{4}$};

\draw(4,4) -- (6,4)(4,4) node[left]  {$\frac{7}{4}$};

\draw(1,2) -- (1,4);

\draw(5,2) -- (5,4);

\draw(1,2) -- (5,2)(5.1,2) node[right] {$h=\frac{3}{2}$};
\draw[dashed](3,2) -- (3,2.5)node[above]{$\gamma_1$};
\draw[dashed](1,4) -- (1,4.5)node[above]{$\gamma_2$};
\draw[dashed](5,4) -- (5,4.5)node[above]{$\gamma_3$};

   \end{tikzpicture}
   \caption{Tree Model\;\,$\M_{1,\phi}$} \label{fig:np3}
 \end{minipage}
 \begin{minipage}[b]{.5\linewidth}
   \centering
   \begin{tikzpicture}[scale=0.7]
    \foreach \x in {0,1,...,4} { \draw (\x, 0) -- (\x,7); } \draw
    (0,0) -- (5,0); \draw plot[mark=*] coordinates {(1,5) (2,3)
    (4,0)};

    \draw plot [mark=thick] coordinates {(1,5) (2,3) (4,0)};

    \draw plot [mark=thick] coordinates {(1,5.01) (2,3.01) (4,0.01)};

     \path (4.7,0) node[above] {$E_0$};

     \path (2.8,1.8) node[right] {$E_1$};

    \path (1.1,4.6) node[right]{$E_2$};

    \path (0.8,6.1) node[right] {$E_3$};

    \draw[dotted] (0.3,7) -- (1,5);

    \draw[dotted] (0,4.5) -- (4,0);

    \path (1.2,2.2)node[right] {$e_1$};

    \draw[dotted] (0,6.5) -- (2,3);

    \path(0.13,4.87)node[right]{$e_2$};

    \path (0.28,6.8) node[right] {$e_3$};

   \draw  [mark=thick] (0.99,5)-- (0.99,7);

     \draw  [mark=thick] (1,5) -- (1,7);

  \draw [mark=thick](4,0)-- (5,0);
\draw  [mark=thick](4,0.01)-- (5,0.01);

   \end{tikzpicture}
   \caption{$\PP(\phi,\zeta_i)$} \label{fig:np4}
 \end{minipage}
 \end{figure}

\bibliographystyle{amsplain}

\begin{thebibliography}{99}
\bibitem{yacoub} Abderrahmane J, Ould M., \emph{Newton Polygon and
Trivialisation of Families}, J. Math. Soc. Japan, 54, 513-550
(2002).
\bibitem{bk} K. Bekka and S. Koike, \emph{ The Kuo condition, an
inequality
 of Thom's type and $(C)$-regularity}, Topology, 
{\bf 37}
(1998), 45-62.

\bibitem{BM} E. Bierstone and P. D. Milman, \emph{Arc-analytic
functions}, Invent. math., {\bf 101} (1990), 411-424.

\bibitem{BM2} E. Bierstone, P. Milman \emph{ Semianalytic and
subanalytic sets},
Inst. Hautes Études Sci. Publ. Math.  No. 67, 1988, 5--42.

\bibitem{BL} J. Bochnak, S. Lojasiwicz
\emph{A Converse of the Kuiper-Kuo Theorem}, in Proc.
Liverpool
Singularities Symposium, I, (ed. C.T.C. Wall), 1971, Lecture Notes
Math., 209, Springer-Verlag.

\bibitem{Brieskorn} E. Brieskorn, H. Kn\"{o}rrer,
\emph{Plane Algebraic Curves}, Birkh\"{a}user Verlag, 1986.

\bibitem{Jonsson}  Charles Favre, Mattias Jonsson, \emph{Valuative
analysis of planar plurisubharmonic functions},
Invent. Math. 162
(2005), No. 2, pp 271--311.

\bibitem{laurent3} T. Fukui and L. Paunescu, \emph{ Modified
Analytic Trivialization for Weighted Homogeneous
Function Germs},
J. Math. Soc. Japan, vol. 52, No. 2, (2000), pp 433--446.

\bibitem{Hancock}H. Hancock, \emph{Theory of Maxima and
Minima}, Dover Publications,
Inc. New York 1960.

\bibitem{dJ} Theo de Jong, Gerhard Pfister,
\emph{Local analytic geometry. Basic theory and
applications},
Advanced Lectures in Mathematics. Friedr. Vieweg \& Sohn,
Braunschweig, 2000.

\bibitem{koike} S. Koike, \emph{ Notes on $C^0$-sufficiency 
of
quasijets},
 J. Math. Soc. Japan, {\bf 42} (1990), 265-275.

\bibitem{kuo} T.-C. Kuo, \emph {On $C^0$-sufficiency of
jets of potential function},
 Topology 8, p.157-171 (1969).
\bibitem{kuo-lu} T.-C. Kuo and Y.C. Lu, \emph{On analytic
function germs of two complex variables}, Topology, 15
(1977),
299--310.

\bibitem{kuo-parusinski} T.-C. Kuo and A. Parusi\'nski,
\emph{Newton Polygon relative to an arc}, in Real and 
Complex
Singularities (S\~ao Carlos, 1998), Chapman \& Hall Res. Notes
Math., {\bf 412}, 2000, 76--93.

\bibitem{kuo-pau} T.-C. Kuo and L. Paunescu,
\emph{Equisingularity in $R^2$ as Morse stability 
in
infinitesimal
calculus}, (Communicated by H. Hironaka) Proc. Japan Acad. Ser. A,
Math. Sci. Volume {\bf 81}, Number 6 (2005), 115-120.

\bibitem{milnor1} J. Milnor, \emph{Morse Theory}, Annals of Math.
Studies. {\bf 51},
Princeton Univ. Press 1983.

\bibitem{milnor} J. Milnor, \emph{Singular Points of
Complex Hypersurfaces}, Annals of Math. Studies. {\bf 61},
Princeton Univ. Press 1968.


\bibitem{laurent} L. Paunescu, \emph{$V$-sufficiency from the
weighted point of view}, J. Math. Soc. Japan, vol. 46, No. 2,
(1994), pp 345--354.

\bibitem{laurent2} L. Paunescu, \emph{A Weighted Version of
the Kuiper-Kuo-Bochnak-Lojasiewicz Theorem}, J. of
Algebraic
Geometry, volume 2, number 1, January (1993), pp 69-79.


\bibitem{Pham} F. Pham,
\emph{Deformations equisingulariti\'es des ideaux
jacobiens de
courbes planes}, in Proc. Liverpool Singularities Sym., II, (ed.
C.T.C. Wall), 1971, 218-233, Lecture Notes Math., 209,
Springer-Verlag.

\bibitem{waerden} van der Waerden, \emph{Einf\"{u}hrung in
die Algebraische Geometrie}, Zweite Auflage. Die Grundlehren 
der
Mathematischen Wissenschaften, Band 51. Springer-Verlag,
Berlin-New York, 1973.

\bibitem{walker} R. J. Walker, \emph{Algebraic Curves},
Springer-Verlag, 1972.

\bibitem{whitney} H. Whitney, \emph{Complex Analytic 
Varieties},
 Addison-Wesley, Reading, Mass., 1972.

\bibitem{algsurface} O. Zariski, \emph{Algebraic Surfaces},
 Ergebnisse der Mathematik und ihrer Grenzgebiete, Band 61.
Springer-Verlag, New York-Heidelberg, 1971.

\bibitem{Zariski} O. Zariski,
\emph{Studies in equisingularity I}, Amer J. Math. 2,
{\bf 87},
(1965), 507-536.
\end{thebibliography}

\end{document}